\documentclass[reqno]{amsart}
\usepackage{amssymb}
\usepackage{amsthm}
\usepackage{amsmath}
\usepackage{enumitem}
\usepackage{xcolor}
\usepackage{booktabs}
\usepackage[foot]{amsaddr}

\usepackage[hyphens]{url}
\usepackage{hyperref}
\usepackage[hyphenbreaks]{breakurl}

\usepackage{orcidlink}
\usepackage{subcaption}

\usepackage{colortbl}
\usepackage{xcolor}

\usepackage{booktabs}
\usepackage{adjustbox}

\hypersetup{
    colorlinks=true,
    linkcolor=babyblue,
    filecolor=babyblue,      
    urlcolor=babyblue,
    citecolor=babyblue
    }

\definecolor{babyblue}{rgb}{0.1, 0.6, 0.75}

\definecolor{orcidlogocol}{rgb}{0.1, 0.6, 0.75}

\newtheorem{theorem}{Theorem}[section]

\newtheorem{proposition}[theorem]{Proposition}
\newtheorem{lemma}[theorem]{Lemma}

\theoremstyle{definition}

\def\Sym{{\rm Sym}}

\def\ker{{\rm ker}\,}
\def\ord{{\rm ord}}

\def\Skew{{\rm Skew}}

\newcommand{\NN}{\mathbb N}
\newcommand{\ZZ}{\mathbb Z}

\newcommand{\OO}{\mathcal{O}}

\begin{document}
\title{Enumeration of skew morphisms of cyclic $2$-groups}
\author[Martin Bachrat{\' y}]{Martin Bachrat{\' y}\,$^{\orcidlink{0000-0002-4300-7507},1,\ast}$}
\address{$^1$Faculty of Civil Engineering, Slovak University of Technology, Bratislava 81005, Slovakia}
\address{$^\ast$Corresponding author}
\email[A1]{martin.bachraty@stuba.sk}
\thanks{The author acknowledges funding from the EU NextGenerationEU through the Recovery and Resilience Plan for Slovakia under the project No. 09I03-03-V04-00272.}

\maketitle

\begin{abstract}
A skew morphism of a finite group $B$ is a permutation of $B$ fixing the identity and satisfying $\varphi(xy) = \varphi(x)\varphi^{i_x}(y)$ for some integers $i_x$ indexed by $x \in B$. The enumeration of skew morphisms of finite cyclic groups remains an open problem. The most substantial progress to date concerns cyclic $p$-groups with $p$ odd, for which a full classification and enumeration was obtained by Kov\'{a}cs and Nedela. In this paper we treat the remaining case $p = 2$, giving a complete classification and enumeration of skew morphisms of finite cyclic $2$-groups. Writing $\mathrm{Skew}(n)$ for the number of skew morphisms of $\mathbb{Z}_n$, we prove that $\mathrm{Skew}(2^e) = 4\,\mathrm{Skew}(2^{e-1}) - 4$ for each $e \geq 4$, and that $\mathrm{Skew}(2^e) = (7 \cdot 4^{e-2} + 8)/6$ for each $e \geq 3$. This completes the enumeration of skew morphisms for all cyclic $p$-groups. \\[0.4em]
\end{abstract}

\section{Introduction}

A \emph{skew morphism} of a finite group $B$ is a permutation $\varphi$ of $B$ fixing the identity element of $B$ and having the property that for each $x\in B$ there exists a non-negative integer $i_x$ such that $\varphi(xy) = \varphi(x)\varphi^{i_x}(y)$ for all $y \in B$. The \emph{order} of $\varphi$, denoted by $\ord(\varphi)$, is defined as the order of the group $\langle \varphi\rangle$. If $\varphi$ is non-trivial, then there is a unique choice for $i_x$ such that $i_x\in \{1,\dots, \ord(\varphi)-1\}$, and the function $\pi_\varphi$ that maps each $x\in B$ to $i_x$ is called the \emph{power function} of $\varphi$. If $\varphi$ is the identity permutation of $B$, we define $\pi_\varphi(x)=1$ for all $x\in B$. Observe that every automorphism of $B$ is a skew morphism with power function identically $1$, as the defining relation then reduces to $\varphi(xy) = \varphi(x)\varphi(y)$, and skew morphisms may thus be regarded as a natural generalisation of automorphisms.

Skew morphisms were introduced more than two decades ago in \cite{JajcaySiran} in the context of Cayley maps. A \emph{Cayley map} $\mathrm{CM}(B, S, \rho)$ is an embedding of a Cayley graph $\mathrm{Cay}(B, S)$ into an orientable surface such
that the cyclic order of arcs based at each vertex is given by a single cyclic permutation $\rho$ of the connection set $S$. Such embeddings are of interest in topological graph theory, where one seeks to understand which graphs admit highly symmetric
surface embeddings. Jajcay and \v{S}ir\'{a}\v{n} showed that a Cayley map $\mathrm{CM}(B, S, \rho)$ admits a map automorphism group regular on arcs if and only if $\rho$ extends to a skew morphism of $B$. This connection transformed
the classification of regular Cayley maps into a problem about the structure of skew morphisms, and provided the original impetus for their systematic study. For further reading on skew morphisms in this context, we refer to \cite{ConderJajcayTucker2007b, ConderJajcayTucker2007, ConderTucker, FengJajcayWang,KovacsKwon}.

Beyond their role in map theory, skew morphisms have emerged as a natural object of independent interest in group theory, owing to a close correspondence with certain group factorisations. Recall that a subgroup $C$ of a group $G$ is \emph{core-free} in $G$ if the largest normal subgroup of $G$ contained in $C$ is trivial. A \emph{complementary product} is a factorisation $G = BC$ where $B$ and $C$ are subgroups of $G$ with $B \cap C = 1$. If in addition $C$ is cyclic and core-free in $G$, we say that $G$ is a \emph{skew product group} for $B$ with a \emph{skew complement} $C$. It was shown in \cite{ConderJajcayTucker2016} that skew morphisms of a finite group $B$ are in a correspondence with skew product groups for $B$. More precisely, if we let $c$ be a generator of $C$, then every element of $G$ writes uniquely as $xc^i$ with $x \in B$ and $i\in \{0,\dots,|C|-1\}$. One then defines a map $\varphi \colon B \to B$ and a function $\pi_\varphi \colon B \to \mathbb{Z}_{|C|}$ by the rule $cx = \varphi(x)\, c^{\pi_\varphi(x)}$ for all $x \in B$; it can be verified that $\varphi$ is a skew morphism of $B$ with power function $\pi_\varphi$. Conversely, if $\varphi$ is a skew morphism of $B$, then $\langle \varphi\rangle$ is a subgroup of $\Sym(B)$, and we may regard $B$ as a group of permutations of itself via left multiplication. The group $G = B\langle \varphi \rangle$ is then a skew product group for $B$ with skew complement $\langle \varphi \rangle$. This correspondence shows that skew morphisms can be used to study group factorisations with a cyclic, core-free complement, placing the subject firmly within the broader theory of group products. More recently, Hu and Jajcay \cite{HuJajcay} have extended this framework beyond the core-free setting, showing that skew morphisms and generalisations of their power functions can also be used to describe complementary products $G = BC$ where $C$ is cyclic but not necessarily core-free in $G$, thereby broadening the scope of the theory.

A central problem in the study of skew morphisms is the classification and enumeration of all skew morphisms for various families of finite groups. This has been achieved for several families: all skew morphisms of finite simple groups were determined in~\cite{BachratyConderVerret}; skew morphisms of dihedral groups were classified in \cite{HuKovacsKwon}; and for cyclic
$p$-groups with $p$ odd, a complete classification and enumeration was obtained in \cite{KovacsNedela2017}. The case of general finite cyclic groups remains open. While a recursive characterisation of all skew morphisms of $\mathbb{Z}_n$ has been given recently in \cite{BachratyHagara2}, this characterisation does not yield a closed formula for the number of skew morphisms of $\mathbb{Z}_n$. 

The correspondence with skew product groups has proven a useful tool in the classifications mentioned above, as illustrated for instance by the classification of skew morphisms of finite simple groups in \cite{BachratyConderVerret}. Yet knowing all skew product groups for a given group $B$ does not in itself resolve the enumeration problem. A single skew product group for $B$ may contain several subgroups isomorphic to $B$, each possibly admitting multiple cyclic core-free complements, and each complement multiple generators, with different choices potentially yielding different skew morphisms. This difficulty is illustrated by the case of cyclic $2$-groups: although all skew product groups for $\mathbb{Z}_{2^e}$ have been completely classified in \cite{DuHuLu}, this has not led to an enumeration of their skew morphisms. In this paper we provide such an enumeration, completing the enumeration of skew morphisms for all cyclic $p$-groups by treating the remaining case $p = 2$. Writing $\mathrm{Skew}(n)$ for the number of skew morphisms of $\mathbb{Z}_n$, our main result is the following.

\begin{theorem}\label{thm:main}
Let $e$ be a positive integer. If $e\geq 4$, then $\mathrm{Skew}(2^e) = 4\,\mathrm{Skew}(2^{e-1}) - 4$.
In particular, for each $e\geq 3$ we have $\mathrm{Skew}(2^e) = \frac{7\cdot 4^{e-2} + 8}{6}$.
\end{theorem}

The proof of Theorem~\ref{thm:main} is motivated by an observation of Conder and Kov{\'a}cs from 2018. Using \textsc{Magma}~\cite{Magma}, they computed skew morphisms of $\mathbb{Z}_{2^e}$ for small values of $e$ and observed that for all but two skew morphisms $\varphi'$ of $\mathbb{Z}_{2^{e-1}}$, the number of skew morphisms of $\mathbb{Z}_{2^e}$ equal to $\varphi'$ when taken modulo $2^{e-1}$ is exactly four. The only exceptions were $\varphi'$ equal to the identity permutation or the automorphism $x \mapsto (2^{e-2}-1)x$, in which case this number is exactly two. A central contribution of this paper is a proof that this is indeed always the case, from which the recursive formula and the closed-form enumeration follow. The author is grateful to Conder and Kov{\'a}cs for communicating this observation.

The paper is organised as follows. In Section~\ref{sec:preli} we provide further background on skew morphisms, and in Section~\ref{sec:prop} we derive properties specific to skew morphisms of cyclic $2$-groups. The proof of Theorem~\ref{thm:main} is then split into three sections: Section~\ref{sec:small} treats the exceptional small cases, Section~\ref{sec:upper} establishes the upper bound $\mathrm{Skew}(2^e) \leq (7\cdot 4^{e-2}+8)/6$, and Section~\ref{sec:exi} completes the proof. We conclude with some remarks in Section~\ref{sec:rem}.

\section{Preliminaries}\label{sec:preli}

Let $\varphi$ be a skew morphism of a finite group $B$. Recall that the power function of $\varphi$ is a mapping $\pi_\varphi$ from $B$ to $\ZZ_{\ord(\varphi)}$ such that $\varphi(xy)=\varphi(x)\varphi^{\pi_\varphi(x)}(y)$ for all $x,y \in B$. The set of all elements of $B$ for which $\pi_\varphi(x)=1$ is called the \emph{kernel} of $\varphi$, denoted by $\ker\varphi$. It can be easily shown that $\ker\varphi$ is a subgroup of $B$, and in~\cite{ConderJajcayTucker2016} it was proved that if $B$ is non-trivial, then so is $\ker\varphi$. It is an easy exercise to show that two elements $x,y\in B$ are in the same right coset of $\ker\varphi$ in $B$ if and only if $\pi_\varphi(x)=\pi_\varphi(y)$; see~\cite{ConderJajcayTucker2007}. In particular, if $B$ is cyclic, then $\ker\varphi$ is also cyclic and generated by the element $|B|/|\ker\varphi|$, and so we have $\pi_\varphi(x)=\pi_\varphi(y)$ if and only if $x\equiv y \pmod{|B|/|\ker\varphi|}$. If $\varphi(\ker\varphi)=\ker\varphi$, then the restriction of $\varphi$ to $\ker\varphi$ is clearly an automorphism of $\ker\varphi$. While in general a skew morphism does not have to preserve its kernel set-wise, it can be shown that $\varphi(\ker\varphi)=\ker\varphi$ if $B$ is abelian; see~\cite{ConderJajcayTucker2007} again. 

Recall that every automorphism of $B$ is a skew morphism of $B$. We say that a skew morphism of $B$ is \emph{proper} if it is not an automorphism of $B$. Equivalently, a skew morphism $\varphi$ of $B$ is proper if and only if $\ker\varphi$ is a proper subgroup of $B$. If $\ord(\varphi)=2$, then by definition we have $\pi_\varphi(x)=1$ for all $x\in B$, and so $\varphi$ is an automorphism. It follows that every proper skew morphism of a finite group has order at least $3$.

Let $\varphi$ be a skew morphism of $\ZZ_n$, and let $m$ be a divisor of $n$ such that for any pair $x,y\in \ZZ_n$ we have $x\equiv y \pmod{m}$ if and only if $\varphi(x)\equiv \varphi(y) \pmod{m}$. Then the skew morphism $\varphi$ \emph{taken modulo} $m$ is the permutation of $\ZZ_m$ defined by $x\mapsto \varphi(x) \bmod{m}$. It can be easily checked that this permutation $\rho$ of $\ZZ_m$ is a skew morphism of $\ZZ_m$ with its power function given by $\pi_\rho(x)=(\pi_\varphi(x) \bmod{\ord(\rho)})$ for all $x\in \ZZ_m$.  

For each skew morphism $\varphi$ of a finite group $B$ we define a function $\sigma_{\varphi}$ from $\NN_0 \times B$ to $\ZZ_{\ord(\varphi)}$ given by $\sigma_{\varphi}(r,x)=\pi_\varphi(x)+\pi_\varphi(\varphi(x))+\dots+\pi_\varphi(\varphi^{r-1}(x))$. It can be shown that $\sigma_{\varphi}(\ord(\varphi),x)=0$, and so if in $\sigma_{\varphi}(r,x)$ we substitute $r$ by any non-negative integer congruent to $r$ modulo $\ord(\varphi)$, the value will remain the same; see~\cite{HuJajcay}. It follows that $\sigma_{\varphi}$ may be viewed as a function from $\ZZ_{\ord(\varphi)} \times B$ to $\ZZ_{\ord(\varphi)}$. In the following lemma we list some useful properties of this function.

\begin{lemma}[\cite{ConderJajcayTucker2016}]\label{lem:sigma}
Let $\varphi$ be a skew morphism of a finite group $B$ with power function $\pi_{\varphi}$, and let $r$ be a positive integer. Then for each $x,y,z\in B$ we have:
\begin{enumerate}[label={\rm (\alph*)},ref=\ref{lem:sigma}(\alph*)]
\item\label{lem:sigma:a} $\varphi^r(xy)=\varphi^r(x)\varphi^{\sigma_{\varphi}(r,x)}(y)$;
\item\label{lem:sigma:b} $\pi_{\varphi}(xy)=\sigma_{\varphi}(\pi_{\varphi}(x),y)$;
\item\label{lem:sigma:c} $\sigma_{\varphi}(r,x)=\pi_{\varphi}^{\, r-1}(x)+\sigma_\varphi(r-1,x)$; and
\item\label{lem:sigma:d} $\varphi^{\pi_{\varphi}(z)}(xy)=\varphi^{\pi_{\varphi}(z)}(x)\varphi^{\pi_{\varphi}(z+x)}(y)$.
\end{enumerate}
\end{lemma}

We now give some facts about orders of skew morphisms that will be used throughout the paper. 

\begin{theorem}[\cite{ConderJajcayTucker2016}]
\label{thm:orderofskew}
The order of a skew morphism of a non-trivial finite group $B$ is less than the order of $B$.
\end{theorem} 

\begin{theorem}[\cite{KovacsNedela2011}]
\label{thm:ordercyclic}
If $\varphi$ is a skew morphism of a group $\ZZ_n$, then $\ord(\varphi)$ divides $n\phi(n)$. Moreover, if $\gcd(\ord(\varphi),n)=1$, then $\varphi$ is an automorphism of $\ZZ_n$.
\end{theorem}

\begin{theorem}[\cite{KovacsNedela2011}]
\label{thm:orderoforbit}
Let $\varphi$ be a skew morphism of a finite group $B$, and let $T$ be an orbit of $\langle \varphi \rangle$. If $T$ generates $B$, then $\ord(\varphi)=|T|$. In particular, if $B=\ZZ_n$, then $\ord(\varphi)$ is equal to the smallest positive integer $i$ satisfying $\varphi^i(1)=1$.
\end{theorem} 

We will also need the following. 

\begin{lemma}[\cite{BachratyJajcay2016}]
\label{lem:power}
Let $\varphi$ be a skew morphism of a finite group $B$, and let $i$ be a positive integer. Then $\varphi^i$ is a skew morphism of $B$ if and only if for every $x\in B$ there exists some $a_x\in \ZZ_{\ord(\varphi)}$ such that $\sigma_{\varphi}(i,x)\equiv ia_x \pmod{\ord(\varphi)}$. If $\varphi^i$ is a skew morphism of $B$, then $\pi_{\varphi^i}(x)$ is the smallest positive solution $a_x$ of the above congruence.
\end{lemma}

\begin{lemma}[\cite{Bachraty}]\label{lem:skewfromorbit}
Let $\varphi$ be a skew morphism of $\ZZ_n$ with power function $\pi_{\varphi}$. Then for each $x\in \ZZ_n$ we have
\[ \varphi(x)=\varphi(1)+\varphi^{\pi_{\varphi}(1)}(1)+\varphi^{\pi_{\varphi}(2)}(1)+ \dots +\varphi^{\pi_{\varphi}(x-1)}(1) \, . \] 
In particular, every skew morphism of $\ZZ_n$ is uniquely determined by its power function and the orbit of $\langle \varphi \rangle$ that contains $1$.
\end{lemma}

\begin{lemma}[\cite{KovacsNedela2011}]\label{lem:derived}
If $\varphi$ is a skew morphism of $\ZZ_n$, then the mapping $\overline{\varphi} \!: \ZZ_{\ord(\varphi)} \to \ZZ_{\ord(\varphi)}$ given by $\overline{\varphi}(x) = \sigma_{\varphi}(x,1)$ is a well-defined skew morphism of $\ZZ_{\ord(\varphi)}$.
\end{lemma}

The skew morphism $\overline{\varphi}$ of $\ZZ_{\ord(\varphi)}$ described in Lemma~\ref{lem:derived} is called the skew morphism \emph{derived} from $\varphi$. The concept of the derived skew morphism was first introduced in~\cite{KovacsNedela2011}, and it was later studied in more detail in~\cite{Bachraty, BachratyHagara, FengHu, HuNedela}. Derived skew morphisms will play a key role in later sections, here we list some of their useful properties. 

\begin{lemma}[\cite{BachratyHagara}]\label{lem:quo}
Let $\varphi$ be a non-trivial skew morphism of $\ZZ_n$, and let $\overline{\varphi}$ be the skew morphism derived from $\varphi$. Then:
\begin{enumerate}[label={\rm (\alph*)},ref=\ref{lem:quo}(\alph*)]
\item\label{lem:quo:a} $\overline{\varphi}^{\, y}(x) = \sigma_{\varphi}(x,y)$; and
\item\label{lem:quo:b} $\pi_{\varphi}(i) = \overline{\varphi}^{\, i}(1)$, so in particular $\ord(\overline{\varphi})=n/|\ker\varphi|$. 
\end{enumerate}
\end{lemma}

The following theorem shows that the quotient of a skew morphism $\varphi$ of a cyclic group can be used to determine which powers of $\varphi$ are skew morphisms as well. 

\begin{theorem}[\cite{BachratyHagara}]\label{thm:power}
Let $\varphi$ be a skew morphism of $\ZZ_n$, and let $r$ be a positive integer. Then $\varphi^r$ is a skew morphism of $\ZZ_n$ if and only if the subgroup of $\ZZ_{\ord(\varphi)}$ generated by $r$ is preserved by $\overline{\varphi}$ set-wise. 
\end{theorem}

\section{Properties of skew morphisms of cyclic $2$-groups}\label{sec:prop}

In this section we provide some useful properties about skew morphisms of cyclic $2$-groups. We start with the following lemma.

\begin{lemma}\label{lem:2basic}
Let $\varphi$ be a skew morphism of $\ZZ_{2^e}$. Then:
\begin{enumerate}[label={\rm (\alph*)},ref=\ref{lem:2basic}(\alph*)]
\item\label{lem:2basic:a} the order of $\varphi$ is a power of $2$;
\item\label{lem:2basic:b} $\varphi$ taken modulo $2^i$ is a well-defined skew morphism of $\ZZ_{2^i}$ for each $i\in\{0,\dots,e\}$;
 \item\label{lem:2basic:c} every subgroup of $\ZZ_{2^e}$ is preserved by $\varphi$ set-wise; and
\item\label{lem:2basic:d} every power of $\varphi$ is a skew morphism of $\ZZ_{2^e}$.
\end{enumerate}
\end{lemma}
\begin{proof}
By Theorem~\ref{thm:ordercyclic} we know that $\ord(\varphi)$ divides $2^e\phi(2^e)$, which is a power of $2$, and so (a) follows. We prove (b) by induction. The claim is clearly true for $i=e$, and so we may assume that $i\leq e-1$. Let $K=2^{e-1}$ and note that $K$ is contained in every non-trivial subgroup of $\ZZ_{2^e}$. Since $\ker\varphi$ is non-trivial, it follows that $K\in \ker\varphi$. Recall that the restriction of $\varphi$ to $\ker\varphi$ is an automorphism of $\ker\varphi$, and since $K$ is the only non-zero element of $\ker\varphi$ such that $K+K=0$, we deduce that $\varphi(K)=K$. Hence for every $x\in \ZZ_{2^e}$ we have $\varphi(x+K)=\varphi(K+x)=\varphi(K)+\varphi(x)=\varphi(x)+\varphi(K)=\varphi(x)+K$. It follows, that $\varphi$ taken modulo $K$ is a well-defined permutation of $\ZZ_{2^{e-1}}$ and consequently a skew morphism of $\ZZ_{2^{e-1}}$. Denote this skew morphism by $\varphi'$. Since $i\leq e-1$, we may now apply the inductive hypothesis to deduce that $\varphi'$ taken modulo $2^i$ is a well-defined skew morphism of $\ZZ_{2^i}$. Note that $\varphi$ taken modulo $2^i$ and $\varphi'$ taken modulo $2^i$ are clearly the same, and so (b) follows. Next, we prove (c). Let $H$ be an arbitrary subgroup of $\ZZ_{2^e}$, and let $2^i$ be the smallest non-trivial element of $H$. By (b) it follows that $\varphi$ taken modulo $2^i$ is a skew morphism of $\ZZ_{2^i}$. In particular, since every skew morphism fixes the identity, we deduce that the set of all elements of $\ZZ_{2^e}$ that are divisible by $2^i$ must be preserved set-wise. But these are exactly the elements of $H$, which proves (c). Finally, to prove (d) first note that by (a) we know that $\ord(\varphi)$ is a power of $2$. Recall that the skew morphism $\overline{\varphi}$ derived from $\varphi$ is a skew morphism of $\ZZ_{\ord(\varphi)}$. This group is a cyclic $2$-group, and so by (c) we deduce that every subgroup of $\ZZ_{\ord(\varphi)}$ is preserved by $\overline{\varphi}$ set-wise. The rest follows easily by Theorem~\ref{thm:power}.    
\end{proof}

Next, we show that the power function of a skew morphism of a cyclic $2$-group must have a specific structure when taken modulo $2$ and modulo $4$.

\begin{lemma}\label{lem:2power}
Let $\varphi$ be a skew morphism of $\ZZ_{2^e}$ with power function $\pi_\varphi$. Then for each $x\in \ZZ_{2^e}$ we have $\pi_\varphi(x)\equiv 1 \pmod{2}$. Moreover, if $x$ is even, then the values $\pi_\varphi(x)$ are all congruent to $1$ modulo $4$ and if $x$ is odd, then the values $\pi_\varphi(x)$ are either all congruent to $1$ modulo $4$ or all congruent to $3$ modulo $4$.  
\end{lemma}
\begin{proof}
Let $k$ denote the smallest non-trivial element of $\ker\varphi$, and let $\OO_1$ denote the orbit of $\langle \overline{\varphi} \rangle$ that contains $1$. Note that Theorem~\ref{thm:orderoforbit} combined with the second part of Lemma~\ref{lem:quo:b} implies that the length of $\OO_1$ is $k$. Also recall that $\pi_\varphi(x)=\pi_\varphi(y)$ if and only if $x\equiv y \pmod{k}$, and so the power function of $\varphi$ is determined by the values $\pi_\varphi(0),\ \pi_\varphi(1),\ \dots,\ \pi_\varphi(k-1)$. By Lemma~\ref{lem:quo:b} we find that these values can be identified with $\OO_1$. By Lemma~\ref{lem:2basic:a} we know that $\ord(\varphi)$ is a power of two, and it follows that $\overline{\varphi}$ is a skew morphism of a cyclic $2$-group. Then Lemma~\ref{lem:2basic:b} implies that $\overline{\varphi}$ taken modulo $2$ (or $4$) is a skew morphism of $\ZZ_2$ (or $\ZZ_4$). Since the only skew morphism of $\ZZ_2$ is the identity permutation, we deduce that $\OO_1$ taken modulo $2$ is $1$-cycle $(1)$, and it follows that $\pi_\varphi(x)$ is always odd. The second part follows from the fact that the only skew morphisms of $\ZZ_4$ are the identity and $(1,3)$, and hence $\OO_1$ taken modulo $4$ is either a $1$-cycle $(1)$ or a $2$-cycle $(1,3)$.  
\end{proof}

The following lemma shows that if a skew morphism of a cyclic $2$-group has order at least $4$, then the orbit of its generator must have a specific structure.

\begin{lemma}\label{lem:half1}
Let $\varphi$ be a skew morphism of $\ZZ_{2^e}$ of order at least $4$. Then $\varphi^{\ord(\varphi)/2}(1)=2^{e-1}+1$. In particular, $\varphi$ taken modulo $2^{e-1}$ has order $\ord(\varphi)/2$. 
\end{lemma}
\begin{proof}
By Lemma~\ref{lem:2basic:d} we know that both $\varphi^{\ord(\varphi)/2}$ and $\varphi^{\ord(\varphi)/4}$ are skew morphisms of $\ZZ_{2^e}$. Denote these skew morphisms by $\varphi_2$ and $\varphi_4$, respectively, and note that $\ord(\varphi_2)=2$ and $\ord(\varphi_4)=4$. Since there are no proper skew morphisms of order $2$, we deduce that $\varphi_2$ is an automorphism, and it follows that $(\varphi_2(1))^2\equiv 1 \pmod{2^e}$. It follows that $\varphi_2(1)$ is equal to either $2^{e-1}-1$, $2^{e-1}+1$, or $2^{e}-1$. We will show that we must have $\varphi_2(1)=2^{e-1}+1$, which will prove our claim.

First, by Lemma~\ref{lem:2basic:b} we know that $\varphi_4$ taken modulo $4$ is a skew morphism of $\ZZ_4$, and since the only skew morphisms of $\ZZ_4$ are the identity and $(1,3)$, we deduce that $\varphi_4^{\, 2}(1)\equiv 1 \pmod{4}$. Clearly $\varphi_2(1)=\varphi_4^{\, 2}(1)$, and so also $\varphi_2(1)$ is congruent to $1$ modulo $4$. By Theorem~\ref{thm:orderofskew} we know that $2^e$ must be strictly larger than the order of $\varphi$, and so $2^e$ is at least $8$. Consequently, we deduce that both $2^{e-1}$ and $2^{e}$ are divisible by $4$. It follows that $2^{e-1}-1\equiv 2^{e}-1\equiv 3\pmod{4}$, and hence $\varphi_2(1)$ must be equal to $2^{e-1}+1$.  
\end{proof}

We remark that the proof of Lemma~\ref{lem:half1} can be reproduced for any generator of $\ZZ_{2^e}$, that is, for any odd element of $\ZZ_{2^e}$. The section concludes with the following observation.

\begin{lemma}\label{lem:half2}
Let $\varphi$ be a skew morphism of $\ZZ_{2^e}$ of order at least $4$. Then the length of the orbit of $\langle \varphi \rangle$ that contains $2$ is strictly smaller than $\ord(\varphi)$. In particular, every orbit of $\langle \varphi \rangle$ containing an even element has length strictly smaller than $\ord(\varphi)$.
\end{lemma}
\begin{proof}
By Lemma~\ref{lem:half1} we know that $\varphi^{\ord(\varphi)/2}(1)=2^{e-1}+1$. Moreover, by Lemma~\ref{lem:quo:a} we deduce that $\sigma_\varphi(\ord(\varphi)/2,1)=\overline{\varphi}(\ord(\varphi)/2)$, and since $\overline{\varphi}$ is a skew morphism of $\ZZ_{\ord(\varphi)}$, it follows from Lemma~\ref{lem:2basic:c} that $\overline{\varphi}(\ord(\varphi)/2)\in \{0, \ord(\varphi)/2\}$. Since $\overline{\varphi}(0)=0$, it follows that $\overline{\varphi}(\ord(\varphi)/2)=\ord(\varphi)/2$, and so $\varphi^{\ord(\varphi)/2}(2) = \varphi^{\ord(\varphi)/2}(1+1) = \varphi^{\ord(\varphi)/2}(1) + \varphi^{\sigma_\varphi(\ord(\varphi)/2,1)}(1) = \varphi^{\ord(\varphi)/2}(1) + \varphi^{\ord(\varphi)/2}(1) =2^{e-1}+1 + 2^{e-1}+1 = 2$. Hence the length of the orbit of $\langle \varphi \rangle$ that contains $2$ is at most $\ord(\varphi)/2$, which proves the first part of the assertion. To prove the second assertion first note that by Lemma~\ref{lem:2basic:c} we know that the restriction of $\varphi$ to $\langle 2 \rangle$ is a skew morphism of $\langle 2 \rangle$. Theorem~\ref{thm:orderoforbit} implies that the order of $\varphi\vert_{\langle 2\rangle}$ is equal to the length of the orbit of $\langle \varphi \rangle$ that contains $2$, which was shown to be strictly smaller than $\ord(\varphi)$. The second part of the assertion now follows easily as every orbit of $\langle \varphi \rangle$ that contains an even element is also an orbit of $\langle\varphi\vert_{\langle 2 \rangle}\rangle$.  
\end{proof}

\section{Proof of Theorem~\ref{thm:main}: Small cases}\label{sec:small}

We will prove Theorem~\ref{thm:main} by showing that for each skew morphism $\varphi'$ of $\ZZ_{2^{e-1}}$ there are exactly four skew morphisms of $\ZZ_{2^e}$ that are equal to $\varphi'$ when taken modulo $2^{e-1}$, the only exceptions being the identity permutation and the automorphism of $\ZZ_{2^{e-1}}$ (of order $2$) given by multiplication by $2^{e-2}-1$. In these two cases there are exactly two skew morphisms of $\ZZ_{2^e}$ that are equal to $\varphi'$ when taken modulo $2^{e-1}$. In this section we address some of the small cases, namely we look at all candidates for $\varphi'$ such that $\ord(\varphi')$ is at most $2$. We start with the following well known fact about square roots of unity modulo a power of two. 

\begin{lemma}[\cite{IrelandRosen}]\label{lem:roots}
Let $e$ be a positive integer. If $e=1$, then the only square root of unity modulo $2^e$ is $1$. If $e=2$, then the only square roots of unity modulo $2^e$ are $1$ and $3$. If $e\geq 3$, then the square roots of unity modulo $2^e$ are exactly $1$, $2^{e-1}-1$, $2^{e-1}+1$, and $2^e-1$. 
\end{lemma}

Next, we will describe all skew morphisms of $\ZZ_{2^e}$ that have order $4$.

\begin{lemma}\label{lem:skeworder4}
Let $e$ be a positive integer such that $e\geq 4$. Then $\ZZ_{2^e}$ admits exactly four automorphisms of order $4$ and four proper skew morphisms of order $4$. The automorphisms are given by multiplication by $2^{e-2}-1$, $2^{e-2}+1$, $2^{e-1}+ 2^{e-2}-1$, and $2^{e-1}+ 2^{e-2}+1$. For each non-negative integer $k$ the proper skew morphisms $\varphi_1$, $\varphi_2$, $\varphi_3$, and $\varphi_4$ are given by:
\begin{align*}
&\varphi_1(x)=\begin{cases}1+2k+2^{e-2} &\text{ if } x=1+2k,\\  2k &\text{ if } x=2k, \end{cases} \\
&\varphi_2(x)=\begin{cases}1+2k+2^{e-2}+2^{e-1} &\text{ if } x=1+2k,\\  2k &\text{ if } x=2k, \end{cases} \\
&\varphi_3(x)=\begin{cases}1+(2^{e-1}-2)(1+k) &\text{ if } x=1+2k,\\  (2^{e-1}-2)k &\text{ if } x=2k, \end{cases} \\
&\varphi_4(x)=\begin{cases}1+(2^{e-1}-2)(1+k)+2^{e-1} &\text{ if } x=1+2k,\\  (2^{e-1}-2)k &\text{ if } x=2k\, . \end{cases}
\end{align*}
\end{lemma}
\begin{proof}
Let $\varphi$ be a skew morphism of $\ZZ_{2^e}$ of order $4$. Then by Lemma~\ref{lem:half1} we have $\varphi^2(1)=1+2^{e-1}$ and by Lemma~\ref{lem:half2} we know that $\varphi^2(2k)=2k$. If $\varphi$ is an automorphism given by $\varphi(x)=tx$, then we must have $t^2 \equiv 1+2^{e-1} \pmod{2^e}$, and hence $t^2 \equiv 1 \pmod{2^{e-1}}$. It follows by Lemma~\ref{lem:roots} that $t \bmod{2^{e-1}}$ must be either $1$, $2^{e-2}-1$, $2^{e-2}+1$, or $2^{e-1}-1$. We can rule out the first and the last case, since all four candidates for $t$ would give an automorphism of order $2$. This leaves exactly four choices for $t$, all of which give an automorphism of order $4$. Namely, $t$ is equal to either $2^{e-2}-1$, $2^{e-2}+1$, $2^{e-1}+2^{e-2}-1$, or $2^{e-1}+2^{e-2}+1$.

Next, assume that $\varphi$ is proper. Since $\ord(\varphi)=4$ and by Lemma~\ref{lem:2power} all power function values are odd, we deduce that the only non-trivial value taken by $\pi_{\varphi}$ is $3$. It follows that $\ker\varphi = \langle 2 \rangle$, and hence $\pi_{\varphi}(1+2k)=3$ and $\pi_{\varphi}(2k)=1$. Let $\varphi(1)=a$, and note that by Lemma~\ref{lem:half1} we know that the orbit of $\langle \varphi \rangle$ that contains $1$ is equal to $(1,a,1+2^{e-1},a+2^{e-1})$. It follows that $\varphi(2)=\varphi(1)+\varphi^3(1)=2a+2^{e-1}=2(a+2^{e-2})$. Since the restriction of $\varphi$ to its kernel is an automorphism, we deduce that $\varphi^2(2)=\varphi(2(a+2^{e-2}))=2(a+2^{e-2})^2$. But we also have $\varphi^2(2)=2$, and so it follows that $(a+2^{e-2})^2 \equiv 1 \pmod{2^{e-1}}$. It follows by Lemma~\ref{lem:roots} that $a+2^{e-2}$ taken modulo $2^{e-1}$ must be either $1$, $2^{e-2}-1$, $2^{e-2}+1$, or $2^{e-1}-1$. Note that if we subtract $2^{e-2}$ from these four values and take them modulo $2^{e-1}$, then we obtain the same values, only permuted. It follows that also $a$ taken modulo $2^{e-1}$ must be either $1$, $2^{e-2}-1$, $2^{e-2}+1$, or $2^{e-1}-1$. If $a\equiv 1 \pmod{2^{e-1}}$, then $a=1$ or $a=1+2^{e-1}$, but the elements of the orbit $(1,a,1+2^{e-1},a+2^{e-1})$ must be all distinct, so this is impossible. We will also rule out the case $a\equiv 2^{e-2}-1 \pmod{2^{e-1}}$. Note that in this case we have $a+2^{e-2} \equiv 2^{e-2}+ 2^{e-2}-1 \equiv -1 \pmod{2^{e-1}}$, and so we have $\varphi(2k)=-2k$. It follows that $\varphi(2^{e-2}-1)=\varphi(2^{e-2}-2+1)=\varphi(2^{e-2}-2)+\varphi(1)=2-2^{e-2}+a$. If $a=2^{e-2}-1$, then $\varphi(2^{e-2}-1)=\varphi(a)=1+2^{e-1}$, but $2-2^{e-2}+a=2-1=1$, contradiction. Similarly, if $a=2^{e-1}+2^{e-2}-1$, then $\varphi(2^{e-2}-1)=\varphi(a+2^{e-1})=1$, but $2-2^{e-2}+a = 1+2^{e-1}$, contradiction. We are now left with only four possible candidates for $a$, namely $2^{e-2}+1$, $2^{e-1}-1$, $2^{e-1}+2^{e-2}+1$, or $2^e-1$. By Lemma~\ref{lem:skewfromorbit} we know that $\varphi$ is uniquely determined by its power function and the orbit $(1,a,1+2^{e-1},a+2^{e-1})$. Since we know the power function of $\varphi$ and there are at most four candidates for $a$, we deduce that there are at most four proper skew morphisms of $\ZZ_{2^e}$. We will show that these are exactly the permutations $\varphi_1$, $\varphi_2$, $\varphi_3$, and $\varphi_4$. In all four cases the power function is given by $\pi_{\varphi}(1+2k)=3$ and $\pi_{\varphi}(2k)=1$. It can be easily seen that all four functions are bijections that fix the identity of $\ZZ_{2^e}$, so we only have to check that for each $\varphi\in \{\varphi_1, \varphi_2, \varphi_3, \varphi_4\}$ we have $\varphi(x+y)=\varphi(x)+\varphi^{\pi(x)}(y)$. 

We will first look at $\varphi_1$. For each pair of non-negative integers $k,\ell$ we have $\varphi_1(1+2k+2\ell)=1+2k+2\ell+2^{e-2}=1+2k+2^{e-2}+2\ell=\varphi_1(1+2k)+\varphi_1(2\ell)=\varphi_1(1+2k)+\varphi_1^{\, 3}(2\ell)$. We also have $\varphi_1(1+2k+1+2\ell)=1+2k+1+2\ell=1+2k+1+2\ell+4\cdot 2^{e-2}=1+2k+2^{e-2}+1+2\ell+3\cdot 2^{e-2} = \varphi_1(1+2k)+\varphi_1^{\, 3}(1+2\ell)$. Next, we have $\varphi_1(2k+1+2\ell)=2k+1+2\ell+2^{e-2}=\varphi_1(2k)+\varphi_1(1+2\ell)$. And finally, we have $\varphi_1(2k+2\ell) = 2k+2\ell = \varphi_1(2k)+\varphi_1(2\ell)$. 
This proves that $\varphi_1$ is a skew morphism of $\ZZ_{2^e}$. Moreover, it can be easily checked that $\varphi_2=(\varphi_1)^3$, and so by Lemma~\ref{lem:2basic:d} we find that also $\varphi_2$ is a skew morphism $\ZZ_{2^e}$. Next, we look at $\varphi_3$. First, we have $\varphi_3(1+2k+2\ell)= 1+(2^{e-1}-2)(1+k+\ell) = 1+(2^{e-1}-2)(1+k) + (2^{e-1}-2)\ell=\varphi_3(1+2k)+\varphi_3(2\ell) =\varphi_3(1+2k)+\varphi_3^{\, 3}(2\ell)$. Next, we have $\varphi_3(1+2k+1+2\ell)=\varphi_3(2(1+k+\ell))= (2^{e-1}-2)(1+k+\ell) = 1 + (2^{e-1}-2)(1+k) - 1 + (2^{e-1}-2)\ell = \varphi_3(1+2k) - 1 + (2^{e-1}-2)\ell$. Since $e\geq 4$, we know that the product of $2^{e-1}$ and $2^{e-2}$ is equal to $0$ in $\ZZ_{2^e}$, and hence it can be easily verified that $\varphi_3(1+2k) - 1 + (2^{e-1}-2)\ell=\varphi_3(1+2k) + 1 + (2^{e-1}-2)(1+\ell+2^{e-2})=\varphi_3(1+2k)+\varphi_3(1+2\ell+2^{e-1})$. Also note that $(\varphi_3)^2(1+2k)=\varphi_3(1+(2^{e-1}-2)(1+k))=\varphi_3(1+2(2^{e-2}-1)(1+k))=1+(2^{e-1}-2)(1+(2^{e-2}-1)(1+k))$, which can be simplified to $1+2k+2^{e-1}$. Using this, we find that $\varphi_3(1+2\ell+2^{e-1})=\varphi_3(\varphi_3^{\, 2}(1+2\ell))=\varphi_3^{\, 3}(1+2\ell)$, and so $\varphi_3(1+2k+1+2\ell)=\varphi_3(1+2k)+\varphi_3^{\, 3}(1+2\ell)$. Next, we have $\varphi_3(2k+1+2\ell)= 1+(2^{e-1}-2)(1+k+\ell) = (2^{e-1}-2)k +1 +(2^{e-1}-2)(1+\ell) = \varphi_3(2k)+\varphi_3(1+2\ell)$. Finally, we have $\varphi_3(2k+2\ell)= (2^{e-1}-2)(k+\ell) = (2^{e-1}-2)k+(2^{e-1}-2)\ell = \varphi_3(2k)+\varphi_3(2\ell)$. This proves that $\varphi_3$ is a skew morphism of $\ZZ_{2^e}$, and since $\varphi_4=(\varphi_3)^3$, by Lemma~\ref{lem:2basic:d} we find that also $\varphi_4$ is a skew morphism of $\ZZ_{2^e}$.
\end{proof}

We now formulate and prove the main observation of this section.

\begin{proposition}\label{prop:small}
Let $\varphi'$ be a skew morphism of $\ZZ_{2^{e-1}}$ with $e\geq 2$ such that $\ord(\varphi')$ is at most $2$. Also let $x_1$ and $x_2$ be elements of $\ZZ_{2^e}$ such that $x_1 \equiv \varphi'(1) \pmod{2^{e-1}}$ and $x_2 \equiv \varphi'(2) \pmod{2^{e-1}}$. If $\varphi'$ is the identity of $\ZZ_{2^{e-1}}$, then the identity permutation of $\ZZ_{2^e}$ and the automorphism $\ZZ_{2^e}$ given by $x\mapsto (2^{e-1}+1)x$ are the only skew morphisms of $\ZZ_{2^e}$ that are equal to $\varphi'$ when taken modulo $2^{e-1}$. If $\varphi'$ is the automorphism of $\ZZ_{2^{e-1}}$ given by $x\mapsto (2^{e-2}-1)x$, then the automorphisms of $\ZZ_{2^e}$ given by $x\mapsto (2^{e-2}-1)x$ and $x\mapsto (2^{e-1}+2^{e-2}-1)x$ are the only skew morphisms of $\ZZ_{2^e}$ that are equal to $\varphi'$ when taken modulo $2^{e-1}$. In all other cases there exists a unique skew morphism $\varphi$ of $\ZZ_{2^e}$ such that $\varphi(1)=x_1$, $\varphi(2)=x_2$, and $\varphi$ taken modulo $2^{e-1}$ is equal to $\varphi'$.
\end{proposition}
\begin{proof}
Let $\varphi$ be a skew morphism of $\ZZ_{2^e}$ such that $\varphi$ taken modulo $2^{e-1}$ is equal to $\varphi'$. Since the order of $\varphi'$ is at most $2$, by Lemma~\ref{lem:half1} we deduce that $\ord(\varphi)$ is at most $4$. Next, since there are no proper skew morphisms of $\ZZ_{2^{e-1}}$ of order $2$, we deduce that $\varphi'$ is an automorphism. Say $\varphi'(x)=tx$ for all $x\in \ZZ_{2^{e-1}}$, and hence $1=(\varphi')^2(1)=t^2$. Consequently, $t$ must be a square root of unity modulo $2^{e-1}$. If $e=2$, then by Lemma~\ref{lem:roots} we deduce that $t=1$ and $\varphi'$ is the identity permutation of $\ZZ_2$. There are exactly two skew morphisms of $\ZZ_4$, namely the identity permutation and the automorphism given by $x\mapsto 3x$. Both of these permutations taken modulo $2$ are equal to $\varphi'$, and so the assertion is true for $e=2$. If $e=3$, then by Lemma~\ref{lem:roots} we have either $t=1$ or $t=3$. There are exactly six skew morphisms of $\ZZ_8$ of order at most $4$, namely the identity permutation, $(1,5)(3,7)$, $(1,3)(2,6)(5,7)$, $(1,7)(2,6)(3,5)$, $(1,3,5,7)$, and $(1,7,5,3)$. The first two permutations taken modulo $4$ are equal to the identity permutation of $\ZZ_4$, while the remaining four permutations taken modulo $4$ are equal to the automorphism of $\ZZ_4$ given by $\varphi'(x)=3x$. Moreover, if $\varphi'(x)=3x$, then for each $x_1\in \{3,7\}$ and $x_2\in \{2,6\}$ there is exactly one choice for $\varphi$ such that $\varphi(1)=x_1$, $\varphi(2)=x_2$, and $\varphi$ taken modulo $4$ is $\varphi'$. This proves the assertion for $e=3$. Next, let $e\geq 4$. By Lemma~\ref{lem:roots} we know that $t$ is equal to either $1$, $2^{e-2}-1$, $2^{e-2}+1$, or $2^{e-1}-1$. Using Lemma~\ref{lem:skeworder4} (and Lemma~\ref{lem:roots} again) we also know that in total there are twelve skew morphisms of $\ZZ_{2^e}$ of order at most $4$, namely the identity permutation, three automorphisms of order $2$ given by multiplication by $2^{e-1}-1$, $2^{e-1}+1$, and $2^e-1$, four automorphisms of order $4$ given by multiplication by $2^{e-2}-1$, $2^{e-2}+1$, $2^{e-1}+ 2^{e-2}-1$, and $2^{e-1}+ 2^{e-2}+1$, and four skew morphisms $\varphi_1$, $\varphi_2$, $\varphi_3$, and $\varphi_4$ defined in the proof of Lemma~\ref{lem:skeworder4}. If $t=1$, then $\varphi$ is either the identity permutation or the automorphism given by multiplication by $2^{e-1}+1$. If $t=2^{e-2}-1$, then $\varphi$ is either the automorphism given by multiplication by $2^{e-2}-1$ or the automorphism given by multiplication by $2^{e-1}+2^{e-2}-1$. If $t=2^{e-2}+1$, then $\varphi$ is either the automorphism given by multiplication by $2^{e-2}+1$, the automorphism given by multiplication by $2^{e-1}+2^{e-2}+1$, the skew morphism $\varphi_1$, or the skew morphism $\varphi_2$. Moreover, it can be easily seen that the pairs $(\varphi(1),\varphi(2))$ for these four permutations are equal to $(2^{e-2}+1,2^{e-1}+2)$, $(2^{e-1}+2^{e-2}+1,2^{e-1}+2)$, $(2^{e-2}+1,2)$, and $(2^{e-1}+2^{e-2}+1,2)$, respectively. Finally, if $t=2^{e-1}-1$, then $\varphi$ is either the automorphism given by multiplication by $2^{e-1}-1$, the automorphism given by multiplication by $2^e-1$, the skew morphism $\varphi_3$, or the skew morphism $\varphi_4$. In this case the pairs $(\varphi(1),\varphi(2))$ for these four permutations are equal to $(2^{e-1}-1,2^e-2)$, $(2^{e}-1,2^e-2)$, $(2^{e-1}-1,2^{e-1}-2)$, and $(2^{e}-1,2^{e-1}-2)$, respectively. This proves the assertion for all $e\geq 4$.
\end{proof}

\section{Proof of Theorem~\ref{thm:main}: Upper bound}\label{sec:upper}

The aim of this section is to prove that if $\varphi'$ is a skew morphism of $\ZZ_{2^{e-1}}$, then there are at most four skew morphisms of $\ZZ_{2^e}$ that are equal to $\varphi'$ when taken modulo $2^{e-1}$. We start with two technical lemmas.

\begin{lemma}\label{lem:plusK}
Let $\varphi$ be a skew morphism of $\ZZ_{2^e}$, let $i$ be a non-negative integer, and let $K\in \{0,2^{e-1}\}$. Then $\varphi^i(x+K)=\varphi^i(x)+K$ for all $x\in \ZZ_{2^e}$. 
\end{lemma}
\begin{proof}
If $K=0$, then the claim is trivial, and so we may assume that $K=2^{e-1}$. By Lemma~\ref{lem:2basic:b} we know that $\varphi$ taken modulo $K$ is a well defined skew morphism of $\ZZ_K$, and it follows that $\varphi^i(x+K)$ must be congruent to $\varphi^i(x)$ modulo $K$. We can rule out the case $\varphi^i(x+K)=\varphi^i(x)$, for this implies $x+K=x$, contradiction. It follows that $\varphi^i(x+K)=\varphi^i(x)+K$, which proves our claim.
\end{proof}

\begin{lemma}\label{lem:dashpow}
Let $\varphi$ be a skew morphism of $\ZZ_{2^e}$ with $e\geq 1$, let $\varphi'$ be $\varphi$ taken modulo $2^{e-1}$, and for each $a\in \ZZ_{2^e}$ let $\overline{a}=a\bmod{2^{e-1}}$. Then $\pi_{\varphi}(x)\equiv \pi_{\varphi'}(\overline{x}) \pmod{\ord(\varphi')}$ for all $x\in \ZZ_{2^e}$.  
\end{lemma}
\begin{proof}
If $\ord(\varphi)\leq 2$, then $\varphi$ and consequently also $\varphi'$ are automorphisms, and the assertion is trivially true. Hence we may assume that $\ord(\varphi)$ is at least $4$. Note that $\varphi(x+1)=\varphi(x)+\varphi^{\pi_{\varphi}(x)}(1)$ and $\varphi'(\overline{x+1})=\varphi'(\overline{x})+(\varphi')^{\pi_{\varphi'}(\overline{x})}(1)$. Since $\varphi(x+1) \equiv \varphi'(\overline{x+1}) \pmod{2^{e-1}}$ and $\varphi(x) \equiv \varphi'(\overline{x}) \pmod{2^{e-1}}$, we also deduce that   $\varphi^{\pi_{\varphi}(x)}(1)\equiv (\varphi')^{\pi_{\varphi'}(\overline{x})}(1) \pmod{2^{e-1}}$. Then, since $\varphi$ taken modulo $2^{e-1}$ is equal to $\varphi'$, this can be rewritten as $(\varphi')^{\pi_{\varphi}(x)}(1)\equiv (\varphi')^{\pi_{\varphi'}(\overline{x})}(1) \pmod{2^{e-1}}$. The assertion now follows easily from Theorem~\ref{thm:orderoforbit}
\end{proof}

The following proposition will be important. 

\begin{proposition}\label{prop:actions}
Let $e$ be a positive integer, let $\varphi'$ be a skew morphism of $\ZZ_{2^{e-1}}$, and let $\varphi$ and $\psi$ be a pair of skew morphisms of $\ZZ_{2^e}$ such that both $\varphi$ and $\psi$ taken modulo $2^{e-1}$ are equal to $\varphi'$. Then $\varphi$ and $\psi$ have same actions on $\langle 4 \rangle$, $\varphi^2$ and $\psi^2$ have same actions on $\langle 2 \rangle$, and $\varphi^4=\psi^4$.
\end{proposition}
\begin{proof}
Suppose to the contrary that the assertion is false, and let $\varphi'$ be a counterexample such that $e$ is minimal and $\ord(\varphi')$ is minimal among all counterexamples for the given $e$. Therefore we may assume that the assertion is true for all skew morphisms of cyclic $2$-groups of order at most $2^{e-2}$, and for all skew morphisms of $\ZZ_{2^{e-1}}$ with order smaller than $\ord(\varphi')$. It can be easily checked that the assertion is true if $\varphi'$ is a skew morphism of $\ZZ_1$ or $\ZZ_2$, and so we may assume that $e-1$ is at least $2$, and consequently $2^e$ is at least $8$. Moreover, if $\ord(\varphi')=1$, then by Lemma~\ref{lem:half1} we deduce that orders of both $\varphi$ and $\psi$ can be at most $2$. Then clearly $\varphi^2=\psi^2$ and $\varphi^4=\psi^4$. Moreover, if skew morphisms $\varphi$ and $\psi$ (of order $2$) are distinct but equal when taken modulo $2^{e-1}$, then either these are automorphisms given by multiplication by $1$ and $2^{e-1}+1$, or automorphisms given by multiplication by $2^{e-1}-1$ and $2^e-1$. In the first case all elements of $\langle 4 \rangle$ are fixed by both automorphisms, while in the second case all elements of $\langle 4 \rangle$ are inverted by both automorphisms. Either way, the actions of $\varphi$ and $\psi$ on $\langle 4 \rangle$ are identical, and so we cannot have a counterexample with $\ord(\varphi')=1$. Hence we may assume that $\ord(\varphi')$ is at least $2$. If $\ord(\varphi)=2$ or $\ord(\psi)=2$, then Proposition~\ref{prop:small} (and its proof) implies that $\varphi'$ is the automorphism given by multiplication by $2^{e-1}-1$, and $\varphi$ and $\psi$ are equal to one of the following skew morphisms: the automorphism given by multiplication by $2^{e-1}-1$, the automorphism given by multiplication by $2^e-1$, the skew morphism $\varphi_3$, or the skew morphism $\varphi_4$. But it can be easily seen, that all of these skew morphisms invert $\langle 4 \rangle$, their second power acts trivially on $\langle 2 \rangle$, and their fourth power is the identity permutation. It follows that we cannot have a counterexample if $\ord(\varphi)$ or $\ord(\psi)$ is equal to $2$, and so we may assume that both of these orders are at least four. Hence by Lemma~\ref{lem:half1} we also have $\ord(\varphi)=\ord(\psi)=2\ord(\varphi')$.  

So far, we have shown that our counterexample $\varphi'$ has order at least $2$, orders of $\varphi$ and $\psi$ are equal to $2\ord(\varphi')$, and $e$ is at least $3$. Next, we look at the skew morphisms $\varphi^2$ and $\psi^2$ taken modulo $2^{e-1}$. They are both equal to the skew morphism $(\varphi')^2$ of $\ZZ_{2^{e-1}}$. Recall that $\ord(\varphi')\geq 2$, and by Lemma~\ref{lem:2basic:a} we also know that $\ord(\varphi')$ is a power of two. It follows that $\ord((\varphi')^2)<\ord(\varphi')$, and so by the assumption of minimality we find that the assertion is true for $(\varphi')^2$. In particular, we deduce that the actions of $\varphi^4$ and $\psi^4$ on $\langle 2 \rangle$ are the same. Similarly, the skew morphisms $\varphi\vert_{\langle 2\rangle}$ and  $\psi\vert_{\langle 2\rangle}$ taken modulo $2^{e-1}$ are both equal to $\varphi'\vert_{\langle 2 \rangle}$. Then since $\varphi'\vert_{\langle 2 \rangle}$ is a skew morphism of $\ZZ_{2^{e-2}}$, the assumption of minimality implies that the assertion is true for $\varphi'\vert_{\langle 2 \rangle}$, and hence the actions of $\varphi$ and $\psi$ on $\langle 8 \rangle$ are the same.

Next, recall that $\ord(\varphi)=\ord(\psi)=2\ord(\varphi')\geq 4$, and so by Lemma~\ref{lem:half2} we find that the orders of $\varphi\vert_{\langle 2\rangle}$ and $\psi\vert_{\langle 2\rangle}$ are at most $\ord(\varphi')$. Hence by Lemma~\ref{lem:2basic:a} we know that both $\ord(\varphi\vert_{\langle 2\rangle})$ and $\ord(\psi\vert_{\langle 2\rangle})$ are divisors of $\ord(\varphi')$. Then by Lemma~\ref{lem:dashpow} we deduce that for each positive integer $k$ and for each $x\in \ZZ_{2^e}$ we have $\varphi^{\pi_{\varphi}(x)}(2k)=\varphi^{\pi_{\varphi'}(\overline{x})}(2k)=\varphi^{\pi_{\psi}(x)}(2k)$, where $\overline{x}$ is $x$ modulo $2^{e-1}$. Next, since $\varphi$ and $\psi$ are equal when taken modulo $2^{e-1}$, it follows that $\psi(2)=\varphi(2)+K$ with $K\in \{ 0, 2^{e-1}\}$. Note that $e\geq 3$, and so $K$ is even. Then, since $\varphi^4$ and $\psi^4$ restricted to $\langle 2 \rangle$ are the same, we may apply Lemma~\ref{lem:plusK} to find that for any non-negative integer $i$ we have $\psi^{4i+1}(2)=\psi^{4i}(\psi(2))=\psi^{4i}(\varphi(2)+K)=\varphi^{4i}(\varphi(2)+K)=\varphi^{4i}(\varphi(2))+K=\varphi^{4i+1}(2)+K$. We will use this fact to show that $\varphi(4)=\psi(4)$. We have $\psi(4)=\psi(2)+\psi^{\pi_{\psi}(2)}(2)$, and by Lemma~\ref{lem:2power} we know that $\pi_{\psi}(2)=4j+1$ for some non-negative integer $j$. It follows that $\psi(4)=\psi(2)+\psi^{4i+1}(2)=\varphi(2)+K+\varphi^{4i+1}(2)+K=\varphi(2)+\varphi^{4i+1}(2)=\varphi(2)+\varphi^{\pi_{\psi}(2)}(2)=\varphi(2)+\varphi^{\pi_{\varphi}(2)}(2)=\varphi(4)$. We now proceed  to show that $\varphi$ and $\psi$ have the same action on $\langle 4 \rangle$. Recall that $\varphi$ and $\psi$ have the same action on $\langle 8 \rangle$, and so it suffice to show that $\varphi(4+8i)=\psi(4+8i)$ for all positive integers $i$. This follows easily as $\psi(4+8i) = \psi(4)+\psi^{\pi_{\psi}(4)}(8i) = \varphi(4)+\varphi^{\pi_{\psi}(4)}(8i) = \varphi(4)+\varphi^{\pi_{\varphi}(4)}(8i) = \varphi(4+8i)$.

Next, we will show that $\varphi^2$ and $\psi^2$ have same actions on $\langle 2 \rangle$. Note that  Lemma~\ref{lem:2basic:c} implies that both $\langle 2 \rangle$ and $\langle 4 \rangle$ are preserved by $\varphi$, and so $\varphi(2)$ is a multiple of $2$, but cannot be a multiple of $4$. It follows that $\varphi(2)=2+4\ell$ for some non-negative integer $\ell$, and $\psi(2)=2+4\ell+K$. Then $\psi^2(2) =\psi(2+4\ell+K)=\psi(2+4\ell)+K=\psi(2)+\psi^{\pi_{\psi}(2)}(4\ell)+K = \varphi(2)+K+\varphi^{\pi_{\psi}(2)}(4\ell)+K = \varphi(2)+\varphi^{\pi_{\varphi}(2)}(4\ell) =  \varphi(2+4\ell)= \varphi^2(2)$.
We proved that $\varphi$ and $\psi$ (and consequently also $\varphi^2$ and $\psi^2$) have the same action on $\langle 4 \rangle$, and so now it suffice to show that $\varphi^2(2+4i)=\psi^2(2+4i)$ for all positive integers $i$. Note that $\psi(2)\equiv \varphi(2) \mod{2^{e-1}}$, and so using Lemma~\ref{lem:sigma:a} and Lemma~\ref{lem:dashpow} we find that $\psi^2(2+4i)=\psi^2(2)+\psi^{\pi_{\psi}(\psi(2))+\pi_{\psi}(2)}(4i)=\varphi^2(2)+\varphi^{\pi_{\psi}(\psi(2))+\pi_{\psi}(2)}(4i)=\varphi^2(2)+\varphi^{\pi_{\varphi'}\left(\overline{\psi(2)}\right)+\pi_{\varphi}(2)}(4i)=\varphi^2(2)+\varphi^{\pi_{\varphi}(\varphi(2))+\pi_{\varphi}(2)}(4i)=\varphi^2(2+4i)$. (Here we also used the fact that $\pi_{\psi}(\psi(2))+\pi_{\psi}(2)$ is a multiple of $2$, which follows from Lemma~\ref{lem:2power}).

As the last step, we will show that $\varphi^4=\psi^4$. Note that by Lemma~\ref{lem:2basic:b} we know that $\varphi$  taken modulo $4$ is a skew morphism of $\ZZ_4$. Since all skew morphisms of $\ZZ_4$ have order at most $2$, it follows that $\varphi^2(1) \equiv  1 \pmod{4}$, and hence $\varphi^2(1)=1+4t$ for some positive integer $t$. It follows that $\psi^2(1)=1+4t+K$, and since $\varphi$ and $\psi$ have the same action on $\langle 4 \rangle$, it follows that $\psi^4(1)=\psi^2(\psi^2(1))=\psi^2(1+4t+K)=\psi^2(1+4t)+K=\psi^2(1)+\psi^{\pi_{\psi}(\psi(1))+\pi_{\psi}(1)}(4t)+K=\varphi^2(1)+K+\varphi^{\pi_{\psi}(\psi(1))+\pi_{\psi}(1)}(4t)+K = \varphi^2(1)+\varphi^{\pi_{\varphi}(\varphi(1))+\pi_{\varphi}(1)}(4t)=\varphi^2(1+4t)=\varphi^4(1)$. Then since $\varphi^4$ and $\psi^4$ have the same action on $\langle 2 \rangle$, it suffice to show that $\varphi^4(1+2i)=\psi^4(1+2i)$. We have $\psi^4(1+2i)=\psi^4(1)+\psi^{\pi_{\psi}(\psi^3(1))+\pi_{\psi}(\psi^2(1))+\pi_{\psi}(\psi(1))+\pi_{\psi}(1)}(2i)$. Note that by Lemma~\ref{lem:2power} we know that    the sum $\pi_{\psi}(\psi^3(1))+\pi_{\psi}(\psi^2(1))+\pi_{\psi}(\psi(1))+\pi_{\psi}(1)$ must be congruent to $0$ modulo $4$, and hence $\psi^{\pi_{\psi}(\psi^3(1))+\pi_{\psi}(\psi^2(1))+\pi_{\psi}(\psi(1))+\pi_{\psi}(1)}(2i)=\varphi^{\pi_{\psi}(\psi^3(1))+\pi_{\psi}(\psi^2(1))+\pi_{\psi}(\psi(1))+\pi_{\psi}(1)}(2i)$. It follows that $\psi^4(1+2i)=\psi^4(1)+\psi^{\pi_{\psi}(\psi^3(1))+\pi_{\psi}(\psi^2(1))+\pi_{\psi}(\psi(1))+\pi_{\psi}(1)}(2i)=\varphi^4(1)+\varphi^{\pi_{\psi}(\psi^3(1))+\pi_{\psi}(\psi^2(1))+\pi_{\psi}(\psi(1))+\pi_{\psi}(1)}(2i)=\varphi^4(1)+\varphi^{\pi_{\varphi}(\varphi^3(1))+\pi_{\varphi}(\varphi^2(1))+\pi_{\varphi}(\varphi(1))+\pi_{\varphi}(1)}(2i)=\varphi^4(1+2i)$. 
\end{proof}

We are now ready to prove the main theorem of this section.

\begin{theorem}\label{thm:max4}
Let $e$ be a positive integer, let $\varphi'$ be a skew morphism of $\ZZ_{2^{e-1}}$, and let $x_1$ and $x_2$ be elements of $\ZZ_{2^e}$ such that $x_1 \equiv \varphi'(1) \pmod{2^{e-1}}$ and $x_2 \equiv \varphi'(2) \pmod{2^{e-1}}$. Then there is at most one skew morphism $\varphi$ of $\ZZ_{2^e}$ such that $\varphi(1)=x_1$, $\varphi(2)=x_2$, and $\varphi$ taken modulo $2^{e-1}$ is equal to $\varphi'$. In particular, there are at most $4$ skew morphisms of $\ZZ_{2^e}$ that are equal to $\varphi'$ when taken modulo $2^{e-1}$.
\end{theorem}
\begin{proof}
For $\ord(\varphi')\leq 2$ the assertion follows from Proposition~\ref{prop:small}, and so we may assume that $\ord(\varphi')$ is at least $4$. Suppose to the contrary that there are at least five distinct skew morphisms of $\ZZ_{2^e}$ that are all equal when taken modulo $2^{e-1}$. Note that for these skew morphisms there are two possible images of $1$ and two possible images of $2$, and so by the pigeonhole principle we find that for at least two of these skew morphisms, say $\varphi$ and $\psi$, we must have $\varphi(1)=\psi(1)$ and $\varphi(2)=\psi(2)$. Since $\ord(\varphi')\geq 4$ by Lemma~\ref{lem:half1} we find that $\ord(\varphi)=\ord(\psi)=2\ord(\varphi')$. Moreover, by Proposition~\ref{prop:actions} we know that $\varphi$ and $\psi$ have same actions on $\langle 4 \rangle$, $\varphi^2$ and $\psi^2$ have same actions on $\langle 2 \rangle$, and $\varphi^4=\psi^4$. Let $i$ be a positive integer, and note that $\varphi(2+4i)=\varphi(2)+\varphi^{\pi_{\varphi}(2)}(4i)=\psi(2)+\varphi^{\pi_{\varphi}(2)}(4i)$. Recall that $\ord(\varphi)=\ord(\psi)=2\ord(\varphi')$, and so we may reproduce the argument from the proof of Proposition~\ref{prop:actions} to deduce that for each non-negative integer $i$ we have $\varphi^{\pi_{\varphi}(2)}(4i)=\varphi^{\pi_{\psi}(2)}(4i)$. Since the actions of $\varphi$ and $\psi$ on $\langle 4 \rangle$ are equal, this can be further rewritten as $\psi^{\pi_{\psi}(2)}(4i)$, and it follows that $\varphi(2+4i)=\psi(2)+\psi^{\pi_{\psi}(2)}(4i)=\psi(2+4i)$. Therefore $\varphi$ and $\psi$ have same actions on $\langle 2 \rangle$. Similarly, we have $\varphi(1+2i)=\varphi(1)+\varphi^{\pi_\varphi(1)}(2i)=\varphi(1)+\varphi^{\pi_\psi(1)}(2i)=\psi(1)+\psi^{\pi_\psi(1)}(2i)$, and so $\varphi$ and $\psi$ must be equal, contradiction.
\end{proof}

The following lemma shows that for each skew morphism $\varphi$ of $\ZZ_{2^e}$ with $\ord(\varphi)\geq 8$ we can construct three new skew morphisms of $\ZZ_{2^e}$ such that all of them (including $\varphi$) are equal when taken modulo $2^{e-1}$.

\begin{lemma}\label{lem:four}
Let $\varphi$ be a skew morphism of $\ZZ_{2^e}$ of order at least $8$, and let $K=2^{e-1}$. Also for each $x\in \ZZ_{2^e}$ let $\overline{x}=x\bmod 4$, and let $\alpha$, $\beta$, and $\gamma$ be permutations of $\ZZ_{2^e}$ defined as follows:
\begin{align*}
\alpha(x) &= 
\begin{cases} \varphi(x)+K, &  \text{if $\overline{x} \in \{1,3\}$,}\\
\varphi(x), & \text{if $\overline{x} \in \{0,2\}$,}\end{cases} \\
\beta(x) &= 
\begin{cases} \varphi(x)+K, &  \text{if $\overline{x} \in \{2,3\}$,}\\
\varphi(x), & \text{if $\overline{x} \in \{0,1\}$,}\end{cases} \\
\gamma(x) &= 
\begin{cases} \varphi(x)+K, &  \text{if $\overline{x} \in \{1,2\}$,}\\
\varphi(x), & \text{if $\overline{x} \in \{0,3\}$.}\end{cases}
\end{align*}
Then $\alpha$, $\beta$, and $\gamma$ are all skew morphisms of $\ZZ_{2^e}$ of order $\ord(\varphi)$. Moreover, if $\varphi$ taken modulo $4$ is equal to $(1,3)$ and $\pi_{\varphi}(1)$ is congruent to $3$ modulo $4$, then $\alpha$, $\beta$, and $\gamma$ have the same power function as $\varphi$. In all other cases $\alpha$ has the same power function as $\varphi$, and power functions of $\beta$ and $\gamma$ are for each $x\in \ZZ_{2^e}$ given by $\pi_{\beta}(x)=\pi_{\gamma}(x)=\pi_{\varphi}(x)$ if $x$ is even, and by $\pi_{\beta}(x)=\pi_{\gamma}(x)=\pi_{\varphi}(x)+\ord(\varphi)/2$ if $x$ is odd. 
\end{lemma}
\begin{proof}
First, we note that $\alpha$, $\beta$, and $\gamma$ are all well defined permutations of $\ZZ_{2^e}$ and $\alpha(0)=\beta(0)=\gamma(0)=\varphi(0)=0$. Hence to prove that these permutations are skew morphisms it is sufficient to prove that for each $\delta \in \{\alpha,\beta,\gamma \}$ we have $\delta(x+y)=\delta(x)+\delta^{\pi_\delta(x)}(y)$. For this, we must look at all possible residue classes of $x$ and $y$ taken modulo $4$. For specific residue classes (and a specific choice of $\delta$) we clearly know the exact values of $\delta(x+y)$ and $\delta(x)$. If we also specify $\varphi$ taken modulo $4$ and $\pi_{\varphi}(1)$ modulo $4$, we will find whether $\pi_{\delta}(x)$ is equal to $\pi_\varphi(x)$ or $\pi_\varphi(x)+\ord(\varphi)/2$. 

Note that if $y$ modulo $4$ is fixed by $\varphi$ taken modulo $4$, then for each positive integer $i$ we have $\delta^i(y)\equiv y\pmod{4}$. It follows that if $\delta(y)=\varphi(y)$, then $\delta^{\pi_\delta(x)}(y)=\varphi^{\pi_\delta(x)}(y)$. On the other hand, if $\delta(y)=\varphi(y)+K$, then $\delta^2(y)=\delta(\varphi(y)+K)=\varphi(\varphi(y)+K)+K$, and by Lemma~\ref{lem:plusK} it follows that $\delta^2(y)=\varphi^2(y)+K+K=\varphi^2(y)$. By Lemma~\ref{lem:2power} we know that $\pi_\varphi(x)$ is odd, and as $\ord(\varphi)$ is at least $8$, we deduce that $\pi_\delta(x)$ is also always odd. Then it follows easily that $\delta^{\pi_\delta(x)}(y)=\varphi^{\pi_\delta(x)}(y)+K$. 

If $y$ modulo $4$ is not fixed by $\varphi$ taken modulo $4$, then $\varphi$ modulo $4$ must be $(1,3)$. Moreover, $y$ modulo $4$ is either $1$ or $3$ and these residue classes are swapped by $\varphi$. In this case it can be easily checked that $\delta^4(y)=\varphi^4(y)$, and the values $\delta(y)$, $\delta^2(y)$ and $\delta^3(y)$ are determined by $y \bmod{4}$ and the choice of $\delta$. For this reason we need to know the values of $\pi_{\delta}$ taken modulo $4$. Since $\ord(\varphi)$ is at least $8$, we deduce that for each $x\in \ZZ_{2^e}$ we have $\pi_{\delta}(x)\equiv \pi_{\varphi}(x) \pmod{4}$. Hence by Lemma~\ref{lem:2power} we deduce that if $x$ is even, then $\pi_{\delta}(x)$ is congruent to $1$ modulo $4$ and if $x$ is odd, then $\pi_{\delta}(x)$ is congruent to either $1$ or $3$, depending on whether $\pi_{\varphi}(x)$ is congruent to $1$ or $3$.  

Finally, note that in the case when $\pi_{\delta}$ is not equal to $\pi_{\varphi}$, then their values differ only for odd inputs, in which case we have $\pi_{\delta}(x)=\pi_{\varphi}(x)+\ord(\varphi)/2$. By Lemma~\ref{lem:half2} we know that the action of $\varphi$ on $\langle 2 \rangle$ has order at most $\ord(\varphi)/2$, and so if $y$ is even, then $\varphi^{\pi_\delta(x)}(y)=\varphi^{\pi_\varphi(x)+\ord(\varphi)/2}(y)=\varphi^{\pi_\varphi(x)}(y)$. If $y$ is odd, then $y$ generates $\ZZ_{2^e}$ and the proof of Lemma~\ref{lem:half1} can be easily reproduced to show that $\varphi^{\ord(\varphi)/2}(y)=K+y$. Then by Lemma~\ref{lem:plusK} we find that $\varphi^{\pi_\delta(x)}(y)=\varphi^{\pi_\varphi(x)+\ord(\varphi)/2}(y)=\varphi^{\pi_\varphi(x)}(y+K)=\varphi^{\pi_\varphi(x)}(y)+K$.       

We have shown that all parts of the equation $\delta(x+y)=\delta(x)+\delta^{\pi_\delta(x)}(y)$ can be rewritten using only $\varphi$, $\pi_{\varphi}$, and $K$. In particular, the left hand side is equal to either $\varphi(x+y)$ or $\varphi(x+y)+K$, while the right hand side is equal to either $\varphi(x)+\varphi^{\pi_\varphi(x)}(y)$ or $\varphi(x)+\varphi^{\pi_\varphi(x)}(y)+K$. Since clearly $\varphi(x+y)=\varphi(x)+\varphi^{\pi_\varphi(x)}(y)$, the equality holds if and only if the number of `$+K$' terms on both sides of the equation is equal. It is now a straightforward but tedious exercise to verify the equality for each $\delta \in \{\alpha, \beta, \gamma\}$, each pair of residue classes of $x$ and $y$ modulo $4$, each value of $\pi_{\varphi}(1)$ modulo $4$, and each choice of $\varphi$ taken modulo $4$. For example, let $\delta=\beta$, let $(x,y)$ be congruent to $(1,3)$ modulo $4$, and let $\varphi$ taken modulo $4$ be the identity permutation of $\ZZ_4$. (As $\varphi$ taken modulo $4$ is the identity permutation, we do not need to specify $\pi_{\varphi}(1)$ modulo $4$.) Then $x+y\equiv 0 \pmod{4}$, and hence $\beta(x+y)=\varphi(x+y)$. We also have $x\equiv 1 \pmod{4}$, and hence $\beta(x)=\varphi(x)$. Next, note that $x$ is odd, and hence $\pi_{\beta}(x)=\pi_{\varphi}(x)+\ord(\varphi)/2$. Then, since $\varphi$ taken modulo $4$ is the identity of $\ZZ_4$ and $y\equiv 3 \pmod{4}$, we have $\beta^{\pi_{\beta}(x)}(y)=\varphi^{\pi_{\beta}(x)}(y)+K=\varphi^{\pi_{\varphi}(x)+\ord{\varphi}/2}(y)+K$. Finally, since $y$ is odd, this can be further rewritten as $\varphi^{\pi_{\varphi}(x)}(y)+K+K=\varphi^{\pi_{\varphi}(x)}(y)$, and it follows that $\beta(x)+\beta^{\pi_{\beta}(x)}(y)=\varphi(x)+\varphi^{\pi_{\varphi}(x)}(y)=\varphi(x+y)=\beta(x+y)$.

It remains to verify that $\ord(\alpha) = \ord(\beta) = \ord(\gamma) = \ord(\varphi)$. Let $\varphi'$ be $\varphi$ taken modulo $2^{e-1}$. Since the order of $\varphi$ is at least $8$, it follows by Lemma~\ref{lem:half1} that $\ord(\varphi')$ is at least $4$. Also note that $\alpha$, $\beta$, $\gamma$, and $\varphi$ are all equal to $\varphi'$, and so (again by Lemma~\ref{lem:half1}) we find that $\ord(\alpha)=\ord(\beta)=\ord(\gamma)=\ord(\varphi)=2\ord(\varphi')$.
\end{proof}

Let $\varphi'$ be a skew morphism of $\ZZ_{2^e}$ such that $\ord(\varphi')\geq 4$. Lemma~\ref{lem:four} together with Theorem~\ref{thm:max4} shows that if there exists a skew morphism $\varphi$ of $\ZZ_{2^e}$ such that $\varphi$ taken modulo $2^{e-1}$ is equal to $\varphi'$, then there are exactly four skew morphisms of $\ZZ_{2^e}$ that are equal to $\varphi'$ when taken modulo $2^{e-1}$. In the following section we prove the existence of $\varphi$.

\section{Proof of Theorem~\ref{thm:main}: Existence}\label{sec:exi}

In this section we finish the proof of Theorem~\ref{thm:main}. We first give the following useful observation.

\begin{lemma}\label{lem:samesame}
Let $\varphi$ and $\psi$ be a pair of skew morphisms of $\ZZ_{2^e}$ that are equal to each other when taken modulo $2^{e-1}$, and equal to the identity permutation of $\ZZ_4$ when taken modulo $4$. If $\varphi(1)=\psi(1)$,  then the orbits of $\langle\varphi\rangle$ and $\langle\psi\rangle$ that contain $1$ are equal.
\end{lemma}
\begin{proof}
Using~\cite{ConderList} it can be easily checked that the assertion is true for $e\leq 2$, and so we may assume that $e$ is at least $3$. If both $\varphi$ and $\psi$ are automorphisms, then the assertion follows easily from the fact that $\varphi(1)=\psi(1)$. Hence we may also assume that at least one of these skew morphisms, say $\varphi$, is proper. Next, suppose that the orders of both $\varphi$ and $\psi$ are at most $4$. Since $\varphi$ is a proper skew morphism and there are no proper skew morphisms of order at most $2$, it follows that $\ord(\varphi)=4$. Moreover, since $\varphi$ taken modulo $4$ is the identity permutation of $\ZZ_4$ and $e\geq 3$, by Lemma~\ref{lem:skeworder4} we deduce that $\varphi\in \{\varphi_1,\varphi_2\}$. In particular, the orbit of $\langle\varphi\rangle$ that contains $1$ is either $(1,1+2^{e-2},1+2^{e-1},1+2^{e-2}+2^{e-1}$ or $(1,1+2^{e-2}+2^{e-1},1+2^{e-1},1+2^{e-2}$. The orbits of automorphisms of $\ZZ_{2^e}$ such that $1$ is mapped to either $1+2^{e-2}$ or $1+2^{e-2}+2^{e-1}$ are also $(1,1+2^{e-2},1+2^{e-1},1+2^{e-2}+2^{e-1}$ or $(1,1+2^{e-2}+2^{e-1},1+2^{e-1},1+2^{e-2})$, and it follows that if $\psi(1)=\varphi(1)$, then the orbits of $\langle\varphi\rangle$ and $\langle\psi\rangle$ that contain $1$ are equal. Finally, suppose that the order of $\varphi$ or $\psi$ is at least $8$, and let $\varphi'$ denote the skew morphism of $\ZZ_{2^{e-1}}$ that is equal to both $\varphi$ and $\psi$ taken modulo $2^{e-1}$. Then by Lemma~\ref{lem:half1} we deduce that $\ord(\varphi')$ is at least $4$, and then (again by Lemma~\ref{lem:half1}) we have $\ord(\varphi)=\ord(\psi)=2\ord(\varphi')$. In particular, the orders of both $\varphi$ and $\psi$ are at least $8$. Note that by Theorem~\ref{thm:max4} we know that there are at most $4$ skew morphism of $\ZZ_{2^e}$ that are equal to $\varphi'$ when taken modulo $2^{e-1}$, and so by Lemma~\ref{lem:four} we find that these are exactly the skew morphism $\varphi$, $\alpha$, $\beta$, and $\gamma$. Then since $\varphi(1)=\psi(1)$, we must have $\psi\in \{\varphi,\beta \}$. If $\psi=\varphi$, then the assertion is trivially true, and so we may assume that $\psi=\beta$. Then, since $\psi$ taken modulo $4$ is the identity permutation and $\psi(x)=\varphi(x)$ for all $x\equiv 1 \pmod{4}$, we deduce that the orbit of $\langle\psi\rangle$ that contains $1$ is equal to the orbit of $\langle\varphi\rangle$ that contains $1$.
\end{proof}

We now prove that for each skew morphism $\varphi'$ of $\ZZ_{2^{e-1}}$ of order at least $4$ there exists a skew morphism of $\ZZ_{2^e}$ that is equal to $\varphi'$ when taken modulo $2^{e-1}$.

\begin{theorem}\label{thm:exi}
Let $\varphi'$ be a skew morphism of $\ZZ_{2^{e-1}}$ of order at least $4$, and let $x_1$ and $x_2$ be elements of $\ZZ_{2^e}$ such that $x_1 \equiv \varphi'(1) \pmod{2^{e-1}}$ and $x_2 \equiv \varphi'(2) \pmod{2^{e-1}}$. Then there exists a skew morphism $\varphi$ of $\ZZ_{2^e}$ such that $\varphi(1)=x_1$, $\varphi(2)=x_2$, and $\varphi$ taken modulo $2^{e-1}$ is equal to $\varphi'$.
\end{theorem}
\begin{proof}
Suppose to the contrary that the assertion is false, and let $\varphi'$ be a counterexample such that $e$ is minimal and $\ord(\varphi')$ is minimal among all counterexamples for the given $e$. To simplify the notation we will write $\pi'$ instead of $\pi_{\varphi'}$. Note that the smallest cyclic $2$-group that admits a skew morphism of order $4$ is $\ZZ_8$, and so $e$ is at least $4$. By Lemma~\ref{lem:2basic:c} we know that the subgroup $\langle 2 \rangle$ of $\ZZ_{2^{e-1}}$ is preserved by $\varphi'$, and it follows that $\varphi'\vert_{\langle 2\rangle}$ is a well-defined skew morphism of $\ZZ_{2^{e-2}}$. Denote this skew morphism of $\ZZ_{2^{e-2}}$ by $\psi'$ and note that $\varphi'(2k)=2\psi'(k)$ for each non-negative integer $k$. We proceed by considering two cases depending on whether the order of $\psi'$ is greater than $2$ or not. 

\medskip\noindent
\textbf{Case (1)}: $\ord(\psi')>2$
\medskip

By Lemma~\ref{lem:2basic:a} we know that $\ord(\psi')\geq 4$. Then since $e-2<e-1$, by the assumption of minimality of $e$ we know that for any choice of $y_1$ and $y_2$ from $\ZZ_{2^{e-1}}$ such that $y_1 \equiv \psi'(1) \pmod{2^{e-2}}$ and $y_2 \equiv \psi'(2) \pmod{2^{e-2}}$ there exists a skew morphism $\psi$ of $\ZZ_{2^{e-1}}$ such that $\psi(1)=y_1$, $\psi(2)=y_2$, and $\psi$ taken modulo $2^{e-2}$ is equal to $\psi'$. Note that by Proposition~\ref{prop:actions} the action of $\psi^4$ on $\ZZ_{2^{e-1}}$ is always the same regardless of the choice of $y_1$ and $y_2$. By Lemma~\ref{lem:2power} we know that $\pi'(2)=1+4i$ for some non-negative integer $i$. We set $y_1=x_2/2$ and $y_2=y_1+(\psi^4)^i(y_1)$, and in what follows we let $\psi$ be the unique skew morphism of $\ZZ_{2^{e-1}}$ such that $\psi(1)=y_1$, $\psi(2)=y_2$, and $\psi$ taken modulo $2^{e-2}$ is equal to $\psi'$. Note here that $2\psi(1)\equiv 2y_1 \equiv x_2 \equiv \varphi'(2) \equiv 2\psi'(1) \pmod{2^{e-2}}$ and $2\psi(2) \equiv 2y_2 \equiv 2y_1 + 2(\psi')^{4i}(y_1) \equiv x_2 + (\varphi')^{4i}(x_2) \equiv \varphi'(2)+ (\varphi')^{4i+1}(2) \equiv \varphi'(2)+ (\varphi')^{\pi'(2)}(2)\equiv \varphi'(4) \equiv 2\psi'(2) \pmod{2^{e-2}}$, and so indeed $\psi$ taken modulo $2^{e-2}$ is equal to $\psi'$.

\medskip\noindent
\textbf{Case (2)}: $\ord(\psi')\leq 2$
\medskip

Since $\psi'=\varphi'\vert_{\langle 2\rangle}$, it follows that $(\varphi')^2(2)=2$. By Lemma~\ref{lem:2power} we know that $\pi'(1)$ is odd, and hence $\varphi'(3)=\varphi'(1)+(\varphi')^{\pi'(1)}(2)=\varphi'(1)+\varphi'(2)$. On the other hand, we have $\varphi'(3)=\varphi'(2)+(\varphi')^{\pi'(2)}(1)$, and so $\varphi'(1)=(\varphi')^{\pi'(2)}(1)$. It follows that $\pi'(2)=1$, and if we write $\pi'(2)=1+4i$, then $i=0$. We claim that similar to the previous case there exists a skew morphism $\psi$ of $\ZZ_{2^{e-1}}$ such that $\psi(1)=x_2/2$, $\psi(2)=\psi(1)+(\psi^4)^i(y_1)$, and $\psi$ taken modulo $2^{e-2}$ is equal to $\psi'$. Note that $\psi(1)=x_2/2$ and $\psi(2)=x_2/2+(\psi^4)^0(x_2/2)=x_2/2+x_2/2=x_2$. It follows easily that we can simply let $\psi$ be the automorphism of $\ZZ_{2^{e-1}}$ given by multiplication by $x_2/2$. Note that the order of $\psi'$ is at most two, and hence $\psi'$ is an automorphism. Since $\psi$ is an automorphism, it follows that $\psi$ taken modulo $2^{e-2}$ is also an automorphism, and it maps $1$ to $x_2/2 \bmod{2^{e-2}}$. Since $x_2/2\equiv \varphi'(2)/2\equiv 2\psi'(1)/2 \equiv \psi'(1) \pmod{2^{e-2}}$, it follows that $\psi$ taken modulo $2^{e-2}$ and $\psi'$ are equal. 

We are now ready to define $\varphi$. For all $x\in \ZZ_{2^e}$ we let: 
\[ \varphi(x)=\begin{cases}
2\psi(k) &\text{ if } x=2k, \\
x_1+2\psi^{\pi'(1)}(k) &\text{ if } x=1+2k\, .\end{cases} \]
Note that $\varphi(1)=x_1+2\psi^{\pi'(1)}(0)=x_1$ and $\varphi(2)=2\psi(1)=2y_1=2(x_2/2)=x_2$. Moreover, we have $2\psi(k) \equiv 2\psi'(k) \equiv \varphi'(2k) \pmod{2^{e-1}}$ and $x_1+2\psi^{\pi'(1)}(k) \equiv \varphi'(1)+(\varphi')^{\pi'(1)}(2k)\equiv \varphi'(1+2k)\pmod{2^{e-1}}$, and so we deduce that $\varphi$ taken modulo $2^{e-1}$ is equal to $\varphi'$. Hence it suffice to prove that $\varphi$ is a skew morphism of $\ZZ_{2^e}$. As the first step, we investigate the power functions of $\varphi'$ and $\psi$. Since $\psi'$ is equal to $\psi$ taken modulo $2^{e-2}$, we know that $\ord(\psi)\in \{\ord(\psi'), 2\ord(\psi')\}$. Moreover, since $\ord(\varphi')$ is at least $4$, by Lemma~\ref{lem:half2} we have $\ord(\varphi')\geq 2\ord(\psi') \geq \ord(\psi)$, and so by Lemma~\ref{lem:2basic:a} we deduce that $\ord(\psi)$ is a divisor of $\ord(\varphi')$. We will show that for each positive integer $x$ we have $\pi_{\psi}(x) \equiv \pi'(2x) \pmod{\ord(\psi)}$. By Lemma~\ref{lem:2power} all power function values are odd, and it follows that the claim is trivially true if $\ord(\psi)\leq 2$. Hence we may assume that $\ord(\psi)$ is at least $4$, and so by Lemma~\ref{lem:half1} we have $\ord(\psi)=2\ord(\psi')$. Recall that $\psi(2)=\psi(1)+\psi^{4i}(\psi(1))=\psi(1)+\psi^{4i+1}(1)=\psi(1)+\psi^{\pi'(2)}(1)$. On the other hand, we have $\psi(2)=\psi(1)+\psi^{\pi_{\psi}(1)}(1)$, and so $\psi^{\pi'(2)}(1)=\psi^{\pi_{\psi}(1)}(1)$. By Theorem~\ref{thm:orderoforbit} we know that the orbit of $\langle\psi\rangle$ that contains $1$ has length $\ord(\psi)$, and hence $\pi'(2)\equiv \pi_{\psi}(1) \pmod{\ord(\psi)}$. Next, using Lemma~\ref{lem:quo:b} we deduce that $\pi_{\psi}(x) = \overline{\psi}^{\, x}(1)$ and $\pi'(2x) = \overline{\varphi'}^{\, 2x}(1) = \left(\overline{\varphi'}^{\, 2}\right)^x(1)$. It follows that to prove our claim we have to show that the orbits of $\overline{\psi}$ and $\overline{\varphi'}^{\, 2}$ (when taken modulo $\ord(\psi)$) that contain $1$ are equal. By Lemma~\ref{lem:quo:a} for each positive integer $x$ we have $\overline{\psi}(x) = \sigma_{\psi}(x,1)=\sum_{0\leq i \leq x-1}\pi_{\psi}(\psi^i(1))$ and $\overline{\varphi'}^{\, 2}(x) = \sigma_{\varphi'}(x,2)=\sum_{0\leq i \leq x-1}\pi'((\varphi')^i(2))$. Then, since $\psi'$ is equal to both $\psi$ taken modulo $2^{e-2}$ and $\varphi'\vert_{\langle 2\rangle}$, we find that both sums taken modulo $\ord(\psi')$ are equal to $\sum_{0\leq i \leq x-1}\pi_{\psi'}((\psi')^i(1))$. It follows that the skew morphisms $\overline{\psi}$ and $\overline{\varphi'}^{\, 2}$ (taken modulo $\ord(\psi)$) of $\ZZ_{\ord(\psi)}$ are both equal when taken modulo $\ord(\psi')$. Recall that $\ord(\psi')=\ord(\psi)/2$, and we also have $\overline{\psi}(1)\equiv \pi_{\psi}(1) \equiv \pi'(2) \equiv \overline{\varphi'}^{\, 2}(1) \pmod{\ord(\psi)}$. Moreover, by Lemma~\ref{lem:2power} we have $\overline{\varphi'}^{\, 2}(1)\equiv \pi'(2) \equiv 1 \pmod{4}$. We also know that $  \pi_{\psi}(1) \equiv \pi'(2) \pmod{\ord(\varphi)}$ and $\ord(\varphi)$ is at least $4$, and so also $\overline{\psi}(1)\equiv \pi_{\psi}(1) \equiv 1 \pmod{4}$. It follows that both $\overline{\psi}$ and $\overline{\varphi'}^{\, 2}$ taken modulo $4$ are equal to the identity, and so we may apply Lemma~\ref{lem:samesame} to deduce that the orbits of $\overline{\psi}$ and $\overline{\varphi'}^{\, 2}$ (taken modulo $\ord(\psi)$) that contain $1$ are equal. This proves that $\pi_{\psi}(x) \equiv \pi'(2x) \pmod{\ord(\psi)}$.

Our next goal is to show that $\ord(\varphi)=2\ord(\varphi')$. We already showed that $\varphi$ taken modulo $2^{e-1}$ is $\varphi'$, and so $\ord(\varphi)\in \{\ord(\varphi'), 2\ord(\varphi')\}$. Note that Lemma~\ref{lem:2basic:c} implies that all even elements of a skew morphism of a cyclic $2$-group are mapped to even elements, and consequently all odd elements must be mapped to odd elements. In particular, since $x_1 \equiv \varphi'(1) \pmod{2^{e-1}}$, we deduce that $x_1$ is odd, and so $x_1=1+2t$ for some non-negative integer $t$. It follows that for each non-negative integer $k$ we have $\varphi^2(1+2k)=\varphi(\varphi(1+2k))=\varphi(x_1+2\psi^{\pi'(1)}(k)) = \varphi(1+2t+2\psi^{\pi'(1)}(k))=x_1+2\psi^{\pi'(1)}(t+\psi^{\pi'(1)}(k))=1+2t+2\psi^{\pi'(1)}(t)+2\psi^{\sigma_{\psi}(\pi'(1),t)+\pi'(1)}(k)$. We have $\sigma_{\psi}(\pi'(1),t) = \sum_{0\leq i \leq \pi'(1)-1}\pi_{\psi}(\psi^i(t))$, and since $\pi_{\psi}(x) \equiv \pi'(2x) \bmod{\ord(\psi)}$, this can be rewritten as $\sum_{0\leq i \leq \pi'(1)-1}\pi'(2\psi^i(t))$. Note here that $\pi'$ take values between $0$ and $2^{e-1}-1$, and so instead of $2\psi^i(t)$ we can write $2(\psi')^i(t)$, which by definition is equal to $(\varphi')^i(2t)$. It follows that $\sigma_{\psi}(\pi'(1),t) = \sum_{0\leq i \leq \pi'(1)-1}\pi'((\varphi')^i(2t)) = \sigma_{\varphi'}(\pi'(1),2t)=\pi'(1+2t)=\pi'(\varphi(1))=\pi'(\varphi'(1))$. Also note that $\varphi^2(1)=\varphi(1+2t)=1+2t+2\psi^{\pi'(1)}(t)$, and therefore $\varphi^2(1+2k)=\varphi^2(1)+2\psi^{\pi'(\varphi'(1))+\pi'(1)}(k)$. For even elements we have simply $\varphi^2(2k)=2\psi^2(k)$. We will now show that $\ord(\varphi^2)=2\ord((\varphi')^2)$, which will imply that $\ord(\varphi)=2\ord(\varphi')$. By Theorem~\ref{thm:max4} we know that $\psi^2$ is the only skew morphism of $\ZZ_{2^{e-1}}$ that is equal to $(\psi')^2$ when taken modulo $2^{e-2}$ and whose images of $1$ and $2$ are equal to $\psi^2(1)$ and $\psi^2(2)$. Let $\rho$ be a skew morphism of $\ZZ_{2^e}$ such that $\rho(1)=\varphi^2(1)$, $\rho(2)=\varphi^2(2)$, and $\rho$ taken modulo $2^{e-1}$ is equal to $(\varphi')^2$. Recall that $\varphi$ taken modulo $2^{e-1}$ is equal to $\varphi'$, and so $\rho(1)$ and $\rho(2)$ are congruent to $(\varphi')^2(1)$ and $(\varphi')^2(2)$. It follows that if $\ord((\varphi')^2)$ is at least $4$, the existence of $\rho$ follows from the assumption of minimality of $\ord(\varphi')$. Also note that by Lemma~\ref{lem:half1} we have $\ord(\rho)=2\ord((\varphi')^2)$. If $\ord((\varphi')^2)$ is smaller than $4$, then we must have $\ord(\varphi')=4$ and $\ord((\varphi')^2)=2$.  Since the order of $\varphi'$ is $4$, by Lemma~\ref{lem:half1} we deduce that $(\varphi')^2$ is the automorphism of $\ZZ_{2^{e-1}}$ given by $x\mapsto (2^{e-2}+1)x$, and so the existence of $\rho$ is provided by Proposition~\ref{prop:small}. Note here that the only permutations of $\ZZ_{2^e}$ of order at most $2$ are the identity and the automorphisms given by $x\mapsto (2^{e-1}-1)x$, $x\mapsto (2^{e-1}+1)x$, and $x\mapsto (2^e-1)x$. Since none of these permutations taken modulo $2^{e-1}$ is equal to $x\mapsto (2^{e-2}+1)x$, it follows that $\rho$ must have order $4$, and so $\ord(\rho)=2\ord((\varphi')^2)$. As the last step, we show that $\rho=\varphi^2$. We defined $\rho$ so that $\rho(2)=\varphi^2(2)=2\psi^2(1)$. We also have $\rho(4)=\rho(2)+\rho^{\pi_{\rho}(2)}(2)=2\psi^2(1)+2\psi^{2\pi_{\rho}(2)}(1)$, and since the length of the orbit of $\rho$ that contains $2$ is at most $\ord(\rho)/2$, by Lemma~\ref{lem:dashpow} we may replace $2\pi_{\rho}(2)$ by $2\pi_{(\varphi')^2}(2)$. By Lemma~\ref{lem:power} we have $2\pi_{(\varphi')^2}(2)=\sigma_{\varphi'}(2,2)=\pi'(\varphi'(2))+\pi'(2)=\pi'(2\psi(1))+\pi'(2)=2(\pi_{\psi}(\psi(1))+\pi_{\psi}(1))=\pi_{\psi^2}(1)$, and so $\rho(4)=2(\psi^2(1)+2\psi^{\pi_{\psi^2}(1)}(2))=2\psi^2(2)$. It follows that $\rho\vert_{\langle 2\rangle}$ is a skew morphism whose images of $1$ and $2$ are $\psi^2(1)$ and $\psi^2(2)$, and $\rho\vert_{\langle 2\rangle}$ taken modulo $2^{e-1}$ is equal to $(\varphi')^2\vert_{\langle 2\rangle}$, which by definition is equal to $(\psi')^2$. Since $\psi^2$ is the unique skew morphism of $\ZZ_{2^{e-1}}$ with these properties, we deduce that $\rho\vert_{\langle 2\rangle}=\psi^2$. Next, for each non-negative integer $k$ we have  $\rho(2k)=2\psi^2(k)=\varphi^2(2k)$ and $\rho(1+2k)=\rho(1)+\rho^{\pi_{\rho}(1)}(2k)=\varphi^2(1)+2(\psi^2)^{\pi_{(\varphi')^2}(1)}(k)=\varphi^2(1)+2\psi^{2\pi_{(\varphi')^2}(1)}(k)=\varphi^2(1)+2\psi^{\pi'(\varphi'(1))+\pi'(1)}(k)=\varphi^2(1+2k)$. It follows that $\varphi^2=\rho$, and consequently $\ord(\varphi)=2\ord(\varphi')$. Moreover, since $\rho$ is a skew morphism, we know that the length of the orbit of $\langle \rho \rangle$ that contains $1$ is $\ord(\rho)$, and it follows easily that the length of the orbit of $\langle\varphi\rangle$ that contains $1$ is $2\ord(\varphi^2)$, which is equal to $\ord(\varphi)$.    

Next, for each $x\in \ZZ_{2^e}$ we define $\pi(x)$ so that $\varphi(x+1)=\varphi(x)+\varphi^{\pi(x)}(1)$. We will first show that this is a well-defined function from $\ZZ_{2^e}$ to $\ZZ_{\ord(\varphi)}$. Let $\OO_1$ and $\OO'_1$ denote the orbits of $\langle \varphi \rangle$ and $\langle \varphi' \rangle$ that contain $1$. The length of $\OO_1$ is $\ord(\varphi)$, so $\pi$ is indeed a function from $\ZZ_{2^e}$ to $\ZZ_{\ord(\varphi)}$. We also know that $\varphi(x+1)\equiv \varphi'(x+1)\equiv \varphi'(x)+(\varphi')^{\pi'(x)}(1) \equiv \varphi(x)+(\varphi')^{\pi'(x)}(1)\pmod{2^{e-1}}$, and so we have either $\varphi(x+1)= \varphi(x)+(\varphi')^{\pi'(x)}(1)$ or $\varphi(x+1)= \varphi(x)+(\varphi')^{\pi'(x)}(1)+K$, where $K=2^{e-1}$. Since the length of $\OO_1$ is $\ord(\varphi)$ and the length of $\OO'_1$ is $\ord(\varphi)/2$, it follows that $\OO_1$ must contain elements $(\varphi')^i(1)$ and $(\varphi')^i(1)+K$ for each $i\in \{1,\dots,\ord(\varphi')\}$. In particular, $\OO_1$ contains both $(\varphi')^{\pi'(x)}(1)$ and $(\varphi')^{\pi'(x)}(1)+K$, and so there exists a unique $\pi(x)$ (in $\ZZ_{\ord(\varphi)}$) such that $\varphi(x+1)=\varphi(x)+\varphi^{\pi(x)}(1)$. We will now show that for any pair of elements $x,y\in \ZZ_{2^e}$ we have $\varphi(x+y)=\varphi(x)+\varphi^{\pi(x)}(y)$.

We clearly have $\pi(x)\equiv \pi'(x) \pmod{\ord(\varphi')}$, and consequently $\pi_{\psi}(x)\equiv \pi(2x) \pmod{\ord(\psi)}$ for all $x\in \ZZ_{2^e}$. It follows that for any pair of non-negative integers $k$ and $\ell$ we have $\varphi(2k)+\varphi^{\pi(2k)}(2\ell)=2(\psi(k)+\psi^{\pi(2k)}(\ell))=2(\psi(k)+\psi^{\pi_{\psi}(k)}(\ell))=2(\psi(k+\ell))=\varphi(2k+2\ell)$. If one of the elements is odd, the situation is more complex. Recall that we have shown that $\varphi^2(1+2k)=\varphi^2(1)+2\psi^{\pi'(\varphi'(1))+\pi'(1)}(k)=\varphi^2(1)+2\psi^{\sigma_{\varphi'}(2,1)}(k)$. We will now use induction to show that for each positive integer $j$ we have $\varphi^j(1+2k)=\varphi^j(1)+2\psi^{\sigma_{\varphi'}(j,1)}(k)$. Using the inductive hypotheses we deduce that 
\begin{align*}
\varphi^j(1+2k) &= \varphi(\varphi^{j-1}(1+2k)) \\
  &=\varphi(\varphi^{j-1}(1)+2\psi^{\sigma_{\varphi'}(j-1,1)}(k)) \\
  &= \varphi(1 + \varphi^{j-1}(1) - 1 + 2\psi^{\sigma_{\varphi'}(j-1,1)}(k) ) \\ 
  &= \varphi\left(1 + 2\left((\varphi^{j-1}(1) - 1)/2 + \psi^{\sigma_{\varphi'}(j-1,1)}(k)\right)\right) \\
  &= \varphi(1) + 2\psi^{\pi'(1)}\left((\varphi^{j-1}(1) - 1)/2 + \psi^{\sigma_{\varphi'}(j-1,1)}(k)\right) \\ 
  &= \varphi(1) + 2\psi^{\pi'(1)}((\varphi^{j-1}(1) - 1)/2) + 2\psi^{\sigma_{\psi}(\pi'(1),(\varphi^{j-1}(1) - 1)/2) +  \sigma_{\varphi'}(j-1,1)}(k)\, .
\end{align*}
By definition we have $\varphi(1) + 2\psi^{\pi'(1)}((\varphi^{j-1}(1) - 1)/2)=\varphi(1+2(\varphi^{j-1}(1) - 1)/2)=\varphi(1+\varphi^{j-1}(1) - 1)=\varphi(\varphi^{j-1}(1))=\varphi^j(1)$. Earlier in the proof we showed that $\sigma_{\psi}(\pi'(1),t)=\sigma_{\varphi'}(\pi'(1),t)$ and the same arguments can be applied here to argue that $\sigma_{\psi}(\pi'(1),(\varphi^{j-1}(1) - 1)/2)=\sigma_{\varphi'}(\pi'(1),\varphi^{j-1}(1) - 1))$. Moreover, since for each $x$ we clearly have $\pi'(\varphi(x))=\pi'(\varphi'(x))$, this can be further rewritten as $\sigma_{\varphi'}(\pi'(1),(\varphi')^{j-1}(1) - 1)$, and then by Lemma~\ref{lem:sigma:b} we find that  $\sigma_{\varphi'}(\pi'(1),(\varphi')^{j-1}(1) - 1)=\pi'(1+(\varphi')^{j-1}(1)-1)=\pi'((\varphi')^{j-1}(1))$. Then by Lemma~\ref{lem:sigma:c} we have $\pi'((\varphi')^{j-1}(1))+\sigma_{\varphi'}(j-1,1)=\sigma_{\varphi'}(j,1)$, and it follows that $\varphi^j(1+2k)=\varphi^j(1)+2\psi^{\sigma_{\varphi'}(j,1)}(k)$ as required. 

Next, let $x$ be an arbitrary element of $\ZZ_{2^e}$ and note that for each non-negative integer $k$ we have $\varphi^{\pi(x)}(1+2k)=\varphi^{\pi(x)}(1)+2\psi^{\sigma_{\varphi'}(\pi(x),1)}(k)$. Since $\pi(x)\equiv \pi'(x) \pmod{\ord(\varphi')}$, by Lemma~\ref{lem:sigma:b} we have $\sigma_{\varphi'}(\pi(x),1)=\sigma_{\varphi'}(\pi'(x),1)=\pi'(x+1)$. Further, since $\ord(\psi)$ divides $\ord(\varphi')$, we deduce that $\pi'(x+1)\equiv\pi(x+1) \pmod{\ord(\psi)}$, and it follows that $\varphi^{\pi(x)}(1+2k)=\varphi^{\pi(x)}(1)+2\psi^{\pi'(x+1)}(k)=\varphi^{\pi(x)}(1)+2\psi^{\pi(x+1)}(k)=\varphi^{\pi(x)}(1)+\varphi^{\pi(x+1)}(2k)$. Since $\sigma_{\varphi'}(1,1)=\pi'(1)$, it can be also easily seen that $\varphi(1+2k)=\varphi(1)+\varphi^{\pi(1)}(2k)$. In the rest of the proof we will repeatedly use this observation, the observation from the previous paragraph, and the fact that for each $x\in \ZZ_{2^e}$ we have $\varphi(x+1)=\varphi(x)+\varphi^{\pi(x)}(1)$.

Let $k$ and $\ell$ be a pair of non-negative integers. We already showed that $\varphi(2k+2\ell)=\varphi(2k)+\varphi^{\pi(2k)}(2\ell)$. Note that we defined $\varphi$ so that $\varphi\vert_{\langle 2\rangle}=\psi$, and so $\pi'$ and $\pi$ are interchangeable in the exponent of $\varphi$ if it acts on $\langle 2 \rangle$. In particular, we have $\varphi^{\pi(x)}(2k+2\ell) = \varphi^{\pi(x)}(2k)+\varphi^{\sigma_{\varphi'}(\pi'(x),2k)}(2\ell)=\varphi^{\pi(x)}(2k)+\varphi^{\pi'(x+2k)}(2\ell)=\varphi^{\pi(x)}(2k)+\varphi^{\pi(x+2k)}(2\ell)$. We now address the three remaining cases. First, note that
$\varphi(1+2k+2\ell) = \varphi(1)+\varphi^{\pi(1)}(2k+2\ell) = \varphi(1) + \varphi^{\pi(1)}(2k)+\varphi^{\pi(1+2k)}(2\ell) = \varphi(1+2k)+\varphi^{\pi(1+2k)}(2\ell)$. Next, we have $\varphi(2k+1+2\ell) = \varphi(1+2k+2\ell) = \varphi(1)+\varphi^{\pi(1)}(2k+2\ell) = \varphi(1) + \varphi^{\pi(1)}(2k)+\varphi^{\pi(1+2k)}(2\ell) = \varphi(1+2k)+\varphi^{\pi(1+2k)}(2\ell) = \varphi(2k+1)+\varphi^{\pi(1+2k)}(2\ell) = \varphi(2k)+\varphi^{\pi(2k)}(1)+\varphi^{\pi(1+2k)}(2\ell) = \varphi(2k)+\varphi^{\pi(2k)}(1+2\ell)$. Finally, we have $\varphi(1+2k+1+2\ell) = \varphi(2+2k+2\ell) = \varphi(2+2k)+\varphi^{\pi(2+2k)}(2\ell) = \varphi(1+2k+1)+\varphi^{\pi(2+2k)}(2\ell) = \varphi(1+2k)+\varphi^{\pi(1+2k)}(1)+\varphi^{\pi(2+2k)}(2\ell) = \varphi(1+2k)+\varphi^{\pi(1+2k)}(1+2\ell)$.

As the last step we show that $\varphi$ is a bijection and $\varphi(0)=0$. Since $\varphi\vert_{\langle 2\rangle}=\psi$ and $\psi$ is a skew morphism, it follows that $\varphi(0)=0$ and $\varphi$ restricted to $\langle 2 \rangle$ is a bijection. It also clearly maps odd elements to odd elements and if $\varphi(1+2k)=\varphi(1+2\ell)$, then $x_1+2\psi^{\pi'(1)}(k)=x_1+2\psi^{\pi'(1)}(\ell)$. Since $\psi$ is a bijection, this implies $k=\ell$, and hence $1+2k=1+2\ell$. This proves that $\varphi$ is a bijection, and consequently a skew morphism of $\ZZ_{2^e}$.
\end{proof}

We are now ready to prove Theorem~\ref{thm:main}.
\begin{proof}
It is well-known that the number of skew morphisms of $\ZZ_8$ is $6$; see~\cite{ConderList} for example. Next, let $e\geq 4$ and note that by Lemma~\ref{lem:2basic:b} we know that every skew morphism $\varphi$ of $\ZZ_{2^e}$ taken modulo $2^{e-1}$ is equal to some skew morphism $\varphi'$ of $\ZZ_{2^{e-1}}$. Note that for $e\geq 4$ the identity mapping of $\ZZ_{2^{e-1}}$ and the automorphism of $\ZZ_{2^{e-1}}$ given by $x\mapsto (2^{e-2}-1)x$ are distinct, and by Proposition~\ref{prop:small} we know that there are exactly four distinct skew morphisms of $\ZZ_{2^e}$ that are equal to either of these automorphisms when taken modulo $2^{e-1}$. Proposition~\ref{prop:small} also implies that for every other $\varphi'$ of order at most two we have exactly four choices for $\varphi$. If $\ord(\varphi')\geq 4$, then we have four distinct choices for elements $x_1,x_2\in \ZZ_{2^e}$ such that $x_1\equiv \varphi'(1)\pmod{2^{e-1}}$ and $x_2\equiv \varphi'(2)\pmod{2^{e-1}}$. By Theorem~\ref{thm:exi} it follows that there are at least four skew morphisms of $\ZZ_{2^e}$ that are equal to $\varphi'$ when taken modulo $2^{e-1}$. On the other hand, by Theorem~\ref{thm:max4} there are at most four skew morphisms with this property, and so we deduce that there are exactly four choices for $\varphi$. It follows that for each $e\geq 4$ we have $\Skew(2^e)=4\,\Skew(2^{e-1})-4$. The second part of the assertion follows easily by induction as $4\,\Skew(2^{e-1})-4= 4(7\cdot 4^{e-3} + 8)/6-4=(7\cdot 4^{e-2} + 32)/6-4=(7\cdot 4^{e-2} + 32-24)/6=(7\cdot 4^{e-2} + 8)/6$.
\end{proof}

\section{Remarks}\label{sec:rem}
The largest available computational census of skew morphisms of finite cyclic groups, covering all cyclic groups of order up to $2000$, is available at \cite{Census}. It includes the values of $\mathrm{Skew}(2^e)$ for all $e \leq 10$, which are consistent with Theorem~\ref{thm:main}.

The proof of Theorem~\ref{thm:exi} yields a recursive description of all skew morphisms of $\mathbb{Z}_{2^e}$ in terms of those of $\mathbb{Z}_{2^{e-1}}$. Compared to the general recursive characterisation for finite cyclic groups given in \cite{BachratyHagara2}, this description is considerably simpler, and crucially, it leads to an enumeration in Theorem~\ref{thm:main}. One natural strategy towards enumerating skew morphisms of all finite cyclic groups is to first handle cyclic $p$-groups, and then determine how $\mathrm{Skew}(nm)$ relates to $\mathrm{Skew}(n)$ and $\mathrm{Skew}(m)$
for coprime $n$ and $m$. This was pursued by Kovács and Nedela across two papers. In \cite{KovacsNedela2011}, they showed that if $\gcd(n,m) = \gcd(n,\phi(m)) = \gcd(\phi(n),m) = 1$, where $\phi$ denotes Euler's totient function --- we call $n$ and $m$ satisfying this condition \emph{$\phi$-coprime} --- then every skew morphism of $\mathbb{Z}_{nm}$ decomposes as a direct product of a skew morphism of $\mathbb{Z}_n$ and a skew morphism of $\mathbb{Z}_m$, giving $\mathrm{Skew}(nm) = \mathrm{Skew}(n)\,\mathrm{Skew}(m)$. Then in \cite{KovacsNedela2017}, they obtained a complete enumeration of skew morphisms of $\mathbb{Z}_{p^e}$ for $p$ odd. Together with the present paper, this settles the enumeration problem for all cyclic $p$-groups. The remaining open question is to determine $\mathrm{Skew}(nm)$ when $n$ and $m$ are coprime but not $\phi$-coprime. The challenge here is twofold: not every direct product of a skew morphism of $\mathbb{Z}_n$ and a skew morphism of $\mathbb{Z}_m$ need be a skew morphism of $\mathbb{Z}_{nm}$, and conversely, there may exist skew morphisms of $\mathbb{Z}_{nm}$ that do not decompose as such a product. A resolution of this problem would complete the enumeration of skew morphisms for all finite cyclic groups.

If $n$ and $m$ are coprime but not $\phi$-coprime, then $\mathrm{Skew}(nm)$ need not equal $\mathrm{Skew}(n)\,\mathrm{Skew}(m)$. A computer search using \cite{Census} shows that for all pairs of coprime positive integers $n$ and $m$ with $nm \leq 2000$, the equality $\mathrm{Skew}(nm) = \mathrm{Skew}(n)\,\mathrm{Skew}(m)$ holds if and only if $n$ and $m$ are $\phi$-coprime. Moreover, when $n$ and $m$ are not $\phi$-coprime, $\mathrm{Skew}(nm)$ almost always strictly exceeds $\mathrm{Skew}(n)\,\mathrm{Skew}(m)$. The only exceptions among pairs with $nm \leq 2000$ are $11\times 25$, $11\times 75$, $25\times 31$, and $25\times 33$, in each of which $\mathrm{Skew}(nm) < \mathrm{Skew}(n)\,\mathrm{Skew}(m)$, with for instance $\mathrm{Skew}(275) = 496 < 680 = \mathrm{Skew}(11)\,\mathrm{Skew}(25)$; note that $825$ admits two such factorisations, $11\times 75$ and $25\times 33$. Beyond the data available for $nm \leq 2000$, only partial answers are known in general. A natural starting point is the case when $n = p^a$ and $m = q^b$ for distinct primes $p < q$. Observe that in this setting, $n$ and $m$ are $\phi$-coprime if and only if $p$ does not divide $q - 1$. Some results are known when $n$ and $m$ are not $\phi$-coprime. It was shown in \cite{KovacsNedela2011} that if $n = p$ and $m = q$ for primes $p$ and $q$ such that $p$ divides $q-1$, then $\mathrm{Skew}(pq) = 2(p-1)(q-1)$, which equals $2\,\mathrm{Skew}(p)\,\mathrm{Skew}(q)$. In \cite{Bachraty}, the case $n = 4$ and $m = q$ for an odd prime $q$ was settled: if $q \equiv 1 \pmod{4}$, then $\mathrm{Skew}(4q) = 6q-6$, which equals $3\,\mathrm{Skew}(4)\,\mathrm{Skew}(q)$; if instead $q \equiv 3 \pmod{4}$, then $\mathrm{Skew}(4q) = 4q-4$, which equals $2\,\mathrm{Skew}(4)\,\mathrm{Skew}(q)$. Beyond these cases, the problem appears to be open. Table~\ref{tab:skew_combined} illustrates the data for small cases, listing values of $\mathrm{Skew}(nm)$ for prime powers $n > m$ with distinct prime bases, where $n \leq 49$ and $m \leq 27$. Each cell contains $\mathrm{Skew}(nm)$ with $\mathrm{Skew}(n)\,\mathrm{Skew}(m)$ in parentheses, except when the two values coincide, in which case only one is listed. Such cells are highlighted in green.

\begin{table}
\centering
\begin{adjustbox}{max width=\textwidth, max totalheight=\textheight}
\setlength{\tabcolsep}{2pt}
\begin{tabular}{c|ccccccccccccccc}
 & $2$ & $3$ & $4$ & $5$ & $7$ & $8$ & $9$ & $11$ & $13$ & $16$ & $17$ & $19$ & $23$ & $25$ & $27$ \\
\hline
$3$ & $4\ (2)$ & --- & --- & --- & --- & --- & --- & --- & --- & --- & --- & --- & --- & --- & --- \\
$4$ & --- & $8\ (4)$ & --- & --- & --- & --- & --- & --- & --- & --- & --- & --- & --- & --- & --- \\
$5$ & $8\ (4)$ & \cellcolor{green!20}$8$ & $24\ (8)$ & --- & --- & --- & --- & --- & --- & --- & --- & --- & --- & --- & --- \\
$7$ & $12\ (6)$ & $24\ (12)$ & $24\ (12)$ & \cellcolor{green!20}$24$ & --- & --- & --- & --- & --- & --- & --- & --- & --- & --- & --- \\
$8$ & --- & $24\ (12)$ & --- & $60\ (24)$ & $72\ (36)$ & --- & --- & --- & --- & --- & --- & --- & --- & --- & --- \\
$9$ & $30\ (10)$ & --- & $60\ (20)$ & \cellcolor{green!20}$40$ & $80\ (60)$ & $180\ (60)$ & --- & --- & --- & --- & --- & --- & --- & --- & --- \\
$11$ & $20\ (10)$ & \cellcolor{green!20}$20$ & $40\ (20)$ & $80\ (40)$ & \cellcolor{green!20}$60$ & $120\ (60)$ & \cellcolor{green!20}$100$ & --- & --- & --- & --- & --- & --- & --- & --- \\
$13$ & $24\ (12)$ & $48\ (24)$ & $72\ (24)$ & \cellcolor{green!20}$48$ & \cellcolor{green!20}$72$ & $180\ (72)$ & $160\ (120)$ & \cellcolor{green!20}$120$ & --- & --- & --- & --- & --- & --- & --- \\
$16$ & --- & $80\ (40)$ & --- & $184\ (80)$ & $240\ (120)$ & --- & $600\ (200)$ & $400\ (200)$ & $552\ (240)$ & --- & --- & --- & --- & --- & --- \\
$17$ & $32\ (16)$ & \cellcolor{green!20}$32$ & $96\ (32)$ & \cellcolor{green!20}$64$ & \cellcolor{green!20}$96$ & $292\ (96)$ & \cellcolor{green!20}$160$ & \cellcolor{green!20}$160$ & \cellcolor{green!20}$192$ & $920\ (320)$ & --- & --- & --- & --- & --- \\
$19$ & $36\ (18)$ & $72\ (36)$ & $72\ (36)$ & \cellcolor{green!20}$72$ & \cellcolor{green!20}$108$ & $216\ (108)$ & $332\ (180)$ & \cellcolor{green!20}$180$ & \cellcolor{green!20}$216$ & $720\ (360)$ & \cellcolor{green!20}$288$ & --- & --- & --- & --- \\
$23$ & $44\ (22)$ & \cellcolor{green!20}$44$ & $88\ (44)$ & \cellcolor{green!20}$88$ & \cellcolor{green!20}$132$ & $264\ (132)$ & \cellcolor{green!20}$220$ & $440\ (220)$ & \cellcolor{green!20}$264$ & $880\ (440)$ & \cellcolor{green!20}$352$ & \cellcolor{green!20}$396$ & --- & --- & --- \\
$25$ & $172\ (68)$ & \cellcolor{green!20}$136$ & $552\ (136)$ & --- & \cellcolor{green!20}$408$ & $1364\ (408)$ & \cellcolor{green!20}$680$ & $496\ (680)$ & \cellcolor{green!20}$816$ & $4184\ (1360)$ & \cellcolor{green!20}$1088$ & \cellcolor{green!20}$1224$ & \cellcolor{green!20}$1496$ & --- & --- \\
$27$ & $264\ (82)$ & --- & $528\ (164)$ & \cellcolor{green!20}$328$ & $584\ (492)$ & $1584\ (492)$ & --- & \cellcolor{green!20}$820$ & $1168\ (984)$ & $5280\ (1640)$ & \cellcolor{green!20}$1312$ & $1628\ (1476)$ & \cellcolor{green!20}$1804$ & \cellcolor{green!20}$5576$ & --- \\
$29$ & $56\ (28)$ & \cellcolor{green!20}$56$ & $168\ (56)$ & \cellcolor{green!20}$112$ & $336\ (168)$ & $420\ (168)$ & \cellcolor{green!20}$280$ & \cellcolor{green!20}$280$ & \cellcolor{green!20}$336$ & $1288\ (560)$ & \cellcolor{green!20}$448$ & \cellcolor{green!20}$504$ & \cellcolor{green!20}$616$ & \cellcolor{green!20}$1904$ & \cellcolor{green!20}$2296$ \\
$31$ & $60\ (30)$ & $120\ (60)$ & $120\ (60)$ & $240\ (120)$ & \cellcolor{green!20}$180$ & $360\ (180)$ & $400\ (300)$ & \cellcolor{green!20}$300$ & \cellcolor{green!20}$360$ & $1200\ (600)$ & \cellcolor{green!20}$480$ & \cellcolor{green!20}$540$ & \cellcolor{green!20}$660$ & $1488\ (2040)$ & $2920\ (2460)$ \\
$32$ & --- & $304\ (152)$ & --- & $680\ (304)$ & $912\ (456)$ & --- & $2280\ (760)$ & $1520\ (760)$ & $2040\ (912)$ & --- & $2952\ (1216)$ & $2736\ (1368)$ & $3344\ (1672)$ & $15464\ (5168)$ & $20064\ (6232)$ \\
$37$ & $72\ (36)$ & $144\ (72)$ & $216\ (72)$ & \cellcolor{green!20}$144$ & \cellcolor{green!20}$216$ & $540\ (216)$ & $664\ (360)$ & \cellcolor{green!20}$360$ & \cellcolor{green!20}$432$ & $1656\ (720)$ & \cellcolor{green!20}$576$ & \cellcolor{green!20}$648$ & \cellcolor{green!20}$792$ & \cellcolor{green!20}$2448$ & $3256\ (2952)$ \\
$41$ & $80\ (40)$ & \cellcolor{green!20}$80$ & $240\ (80)$ & $320\ (160)$ & \cellcolor{green!20}$240$ & $740\ (240)$ & \cellcolor{green!20}$400$ & \cellcolor{green!20}$400$ & \cellcolor{green!20}$480$ & $2040\ (800)$ & \cellcolor{green!20}$640$ & \cellcolor{green!20}$720$ & \cellcolor{green!20}$880$ & $1984\ (2720)$ & \cellcolor{green!20}$3280$ \\
$43$ & $84\ (42)$ & $168\ (84)$ & $168\ (84)$ & \cellcolor{green!20}$168$ & $504\ (252)$ & $504\ (252)$ & $560\ (420)$ & \cellcolor{green!20}$420$ & \cellcolor{green!20}$504$ & $1680\ (840)$ & \cellcolor{green!20}$672$ & \cellcolor{green!20}$756$ & \cellcolor{green!20}$924$ & \cellcolor{green!20}$2856$ & $4088\ (3444)$ \\
$47$ & $92\ (46)$ & \cellcolor{green!20}$92$ & $184\ (92)$ & \cellcolor{green!20}$184$ & \cellcolor{green!20}$276$ & $552\ (276)$ & \cellcolor{green!20}$460$ & \cellcolor{green!20}$460$ & \cellcolor{green!20}$552$ & $1840\ (920)$ & \cellcolor{green!20}$736$ & \cellcolor{green!20}$828$ & $2024\ (1012)$ & \cellcolor{green!20}$3128$ & \cellcolor{green!20}$3772$ \\
$49$ & $522\ (222)$ & $1044\ (444)$ & $1044\ (444)$ & \cellcolor{green!20}$888$ & --- & $3132\ (1332)$ & $3332\ (2220)$ & \cellcolor{green!20}$2220$ & \cellcolor{green!20}$2664$ & $10440\ (4440)$ & \cellcolor{green!20}$3552$ & \cellcolor{green!20}$3996$ & \cellcolor{green!20}$4884$ & \cellcolor{green!20}$15096$ & $23924\ (18204)$ \\
\end{tabular}
\end{adjustbox}
\caption{Values of $\mathrm{Skew}(nm)$ for prime powers $n > m$ with distinct prime bases.}
\label{tab:skew_combined}
\end{table}

\bibliographystyle{amsplain-my}
\bibliography{skew2}

\providecommand{\bysame}{\leavevmode\hbox to3em{\hrulefill}\thinspace}
\providecommand{\MR}{\relax\ifhmode\unskip\space\fi MR }
\providecommand{\MRhref}[2]{%
  \href{http://www.ams.org/mathscinet-getitem?mr=#1}{#2}
}
\providecommand{\href}[2]{#2}
\begin{thebibliography}{10}

\bibitem{Bachraty}
M.~Bachratý, \emph{Quotients of skew morphisms of cyclic groups}, Ars Math.
  Contemp. \textbf{24} (2024), \#P2.08,
  \url{https://doi.org/10.26493/1855-3974.2947.cd6}.

\bibitem{BachratyConderVerret}
M.~Bachratý, M.~Conder, and G.~Verret, \emph{Skew product groups for
  monolithic groups}, Algebr. Combin. \textbf{5} (2022), 785--802,
  \url{https://doi.org/10.5802/alco.206}.

\bibitem{Census}
M.~Bachratý and M.~Hagara, \emph{A census of skew-morphisms for cyclic groups
  of small orders}, \url{https://www.math.sk/bachraty/lists/skew.html}.

\bibitem{BachratyHagara}
M.~Bachratý and M.~Hagara, \emph{Observations about skew morphisms of cyclic
  groups}, J. Algebr. Combin. \textbf{61} (2025), \#5,
  \url{https://doi.org/10.1007/s10801-024-01371-6}.

\bibitem{BachratyHagara2}
M.~Bachratý and M.~Hagara, \emph{Recursive characterisation of skew morphisms
  of finite cyclic groups}, 2025, \url{https://arxiv.org/abs/2506.11626}.

\bibitem{BachratyJajcay2016}
M.~Bachratý and R.~Jajcay, \emph{Powers of skew-morphisms}, Symmetries in
  Graphs, Maps, and Polytopes (2016), 1--25,
  \url{https://doi.org/10.1007/978-3-319-30451-9_1}.

\bibitem{Magma}
W.~Bosma, J.~Cannon, and C.~Playoust, \emph{The {M}agma algebra system. {I}.
  {T}he user language}, J. Symbolic Comput. \textbf{24} (1997), 235--265,
  \url{https://doi.org/10.1006/jsco.1996.0125}.

\bibitem{ConderList}
M.~Conder, \emph{List of skew-morphisms for small cyclic groups},
  \url{https://www.math.auckland.ac.nz/~conder/SkewMorphisms-SmallCyclicGroups-60.txt}.

\bibitem{ConderJajcayTucker2007b}
M.~Conder, R.~Jajcay, and T.~W. Tucker, \emph{Regular {C}ayley maps for finite
  abelian groups}, J. Algebr. Combin. \textbf{25} (2007), 343--364,
  \url{https://doi.org/10.1007/s10801-006-0037-0}.

\bibitem{ConderJajcayTucker2007}
M.~Conder, R.~Jajcay, and T.~W. Tucker, \emph{Regular t-balanced {C}ayley
  maps}, J. Combin. Theory Ser. B \textbf{97} (2007), 453--473,
  \url{https://doi.org/10.1016/j.jctb.2006.07.008}.

\bibitem{ConderJajcayTucker2016}
M.~Conder, R.~Jajcay, and T.~W. Tucker, \emph{Cyclic complements and skew
  morphisms of groups}, J. Algebra \textbf{453} (2016), 68--100,
  \url{https://doi.org/10.1016/j.jalgebra.2015.12.024}.

\bibitem{ConderTucker}
M.~Conder and T.~W. Tucker, \emph{Regular {C}ayley maps for cyclic groups},
  Trans. Amer. Math. Soc. \textbf{336} (2014), 3585--3609,
  \url{https://doi.org/10.1090/S0002-9947-2014-05933-3}.

\bibitem{DuHuLu}
S.~Du, K.~Hu., and A.~Lucchini, \emph{Skew-morphisms of cyclic 2-groups}, J.
  Group Theory \textbf{22} (2019), 617--635,
  \url{https://doi.org/10.1515/jgth-2019-2046}.

\bibitem{FengJajcayWang}
R.~Feng, R.~Jajcay, and Y.~Wang, \emph{Regular t-balanced {C}ayley maps for
  abelian groups}, Discrete Math. \textbf{311} (2011), 2309--2316,
  \url{https://doi.org/10.1016/j.disc.2011.04.012}.

\bibitem{FengHu}
Y.-Q. Feng, K.~Hu, R.~Nedela, M.~Škoviera, and N.-E. Wang, \emph{Complete
  regular dessins and skew-morphisms of cyclic groups}, Ars Math. Contemp.
  \textbf{18} (2020), 289--307,
  \url{https://doi.org/10.26493/1855-3974.1748.ebd}.

\bibitem{HuJajcay}
K.~Hu and R.~Jajcay, \emph{Cyclic complementary extensions and skew-morphism},
  J. Group Theory \textbf{29} (2026), 301--324,
  \url{https://doi.org/10.1515/jgth-2024-0144}.

\bibitem{HuKovacsKwon}
K.~Hu, I.~Kovács, and Y.~S. Kwon, \emph{A classification of skew morphisms of
  dihedral groups}, J. Group Theory \textbf{26} (2023), 547--569,
  \url{https://www.degruyter.com/document/doi/10.1515/jgth-2022-0085/}.

\bibitem{HuNedela}
K.~Hu, R.~Nedela, N.-E. Wang, and K.~Yuan, \emph{Reciprocal skew morphisms of
  cyclic groups}, Acta Math. Univ. Comenian. \textbf{88} (2019), 305--318,
  \url{http://www.iam.fmph.uniba.sk/amuc/ojs/index.php/amuc/article/view/1006}.

\bibitem{IrelandRosen}
K.~Ireland and M.~Rosen, \emph{A classical introduction to modern number
  theory}, 2 ed., Graduate Texts in Mathematics, vol.~84, Springer, New York,
  1990.

\bibitem{JajcaySiran}
R.~Jajcay and J.~Širáň, \emph{Skew-morphisms of regular {C}ayley maps},
  Discrete Math. \textbf{244} (2002), 167--179,
  \url{https://doi.org/10.1016/S0012-365X(01)00081-4}.

\bibitem{KovacsKwon}
I.~Kovács and Y.~S. Kwon, \emph{Regular {C}ayley maps for dihedral groups}, J.
  Combin. Theory Ser. B \textbf{148} (2021), 84--124,
  \url{https://doi.org/10.1016/j.jctb.2020.12.002}.

\bibitem{KovacsNedela2011}
I.~Kovács and R.~Nedela, \emph{Decomposition of skew-morphisms of cyclic
  groups}, Ars Math. Contemp. \textbf{4} (2011), 329--349,
  \url{https://doi.org/10.26493/1855-3974.157.fc1}.

\bibitem{KovacsNedela2017}
I.~Kovács and R.~Nedela, \emph{Skew-morphisms of cyclic $p$-groups}, J. Group
  Theory \textbf{20} (2017), 1135--1154,
  \url{https://doi.org/10.1515/jgth-2017-0015}.

\end{thebibliography}

\end{document}